\documentclass[a4paper,10pt, oneside]{article}
\usepackage{amsmath,amssymb,amsthm,bm,mathrsfs,graphicx,siunitx}
\usepackage{authblk}
\usepackage{enumerate}
\usepackage[font=small,labelfont=md,textfont=it]{caption}
\usepackage{floatrow,float}
\usepackage[titletoc, title]{appendix}
\usepackage[colorlinks,linkcolor=blue,citecolor=blue]{hyperref}
\usepackage{etoolbox}
\usepackage{longtable}
\usepackage{diagbox}
\usepackage{booktabs,makecell,multirow}
\usepackage[capitalise,nosort]{cleveref}
\usepackage{subeqnarray,cases,color}
\crefname{equation}{}{}
\crefname{lemma}{Lemma}{Lemmas}
\crefname{theorem}{Theorem}{Theorems}
\crefname{discr}{Discretization}{Discretizations}

\apptocmd{\sloppy}{\hbadness 10000\relax}{}{}

\newcommand{\dual}[1]{\langle {#1} \rangle}

\newcommand{\nm}[1]{\lVert {#1} \rVert}
\newcommand{\Nm}[1]{\left\lVert {#1} \right\rVert}
\newcommand{\nmb}[1]{\big\lVert {#1} \big\rVert}
\newcommand{\nmB}[1]{\Big\lVert {#1} \Big\rVert}

\newcommand{\ssnm}[1]
{
  \left\vert\kern-0.25ex
  \left\vert\kern-0.25ex
  \left\vert
  {#1}
  \right\vert\kern-0.25ex
  \right\vert\kern-0.25ex
  \right\vert
}

\newcommand{\ssnmb}[1]
{
  \big\vert\kern-0.25ex
  \big\vert\kern-0.25ex
  \big\vert
  {#1}
  \big\vert\kern-0.25ex
  \big\vert\kern-0.25ex
  \big\vert
}
\newcommand{\ssnmB}[1]
{
  \Big\vert\kern-0.25ex
  \Big\vert\kern-0.25ex
  \Big\vert
  {#1}
  \Big\vert\kern-0.25ex
  \Big\vert\kern-0.25ex
  \Big\vert
}

\makeatletter
\def\spher@harm#1{%
  \vbox{\hbox{%
    \offinterlineskip
    \valign{&\hb@xt@2\p@{\hss$##$\hss}\vskip.2ex\cr#1\crcr}%
  }\vskip-.36ex}%
}
\def\gshone{\spher@harm{.}}
\def\gshtwo{\spher@harm{.&.}}
\def\gshthree{\spher@harm{.&.&.}}
\let\gsh\spher@harm
\makeatother

\newtheorem{lemma}{Lemma}[section]
\newtheorem{remark}{Remark}[section]
\newtheorem{theorem}{Theorem}[section]

\makeatletter\def\@captype{table}\makeatother

\begin{document}

\title{Convergence of a spatial semi-discretization for a backward
semilinear stochastic parabolic equation\thanks{This work was supported
in part by the National Natural Science Foundation of China   (11901410, 12171340) and
the Fundamental Research Funds for the Central Universities
in China  (2020SCU12063).}
}

\author{Binjie Li\thanks{libinjie@scu.edu.cn}}
\author{Xiaoping Xie\thanks{Corresponding author: xpxie@scu.edu.cn}}
\affil{School of Mathematics, Sichuan University, Chengdu 610064, China}


\date{}
\maketitle

\begin{abstract} 
  This paper studies the convergence of a spatial semi-discretization for a
  backward semilinear stochastic parabolic equation. The filtration is general,
  and the spatial semi-discretization uses the standard continuous piecewise
  linear finite element method. Firstly, higher regularity of the solution to
  the continuous equation is derived. Secondly, the first-order spatial accuracy
  is derived for the spatial semi-discretization. Thirdly, an application of the
  theoretical results to a stochastic linear quadratic control problem is
  presented.
\end{abstract}



\section{Introduction}
Let $ (\Omega, \mathcal F, \mathbb P) $ be a given complete probability space
with a normal filtration $ \mathbb F = \{\mathcal F_t\}_{t \geqslant 0 } $ (i.e.,
$ \mathbb F $ is right continuous and $ \mathcal F_0 $ contains all $ \mathbb P
$-null sets of $ \mathcal F $). Assume that $ W(\cdot) $ is an $ \mathbb F
$-adapted one-dimensional real Brownian motion. Let $ \mathcal O \subset \mathbb
R^d $ ($d = 1, 2,3$) be a bounded convex polygonal domain, and let $ \Delta $ be
the realization of the Laplace operator with homogeneous Dirichlet boundary
condition in $ L^2(\mathcal O) $. We consider the following backward semilinear
stochastic parabolic equation:
\begin{equation}
  \label{eq:model}
  \begin{cases}
    \mathrm{d}p(t) = -( \Delta p(t) + f(t,p(t),z(t)) ) \, \mathrm{d}t +
    z(t) \, \mathrm{d}W(t), \quad 0 \leqslant t \leqslant T, \\
    p(T) = p_T,
  \end{cases}
\end{equation}
where $ p_T \in L^2(\Omega, \mathcal F_T, \mathbb P; \dot H^1) $ and $ f $
satisfies the following conditions:
\begin{itemize}
  \item $ f(\cdot, p, z) \in L_\mathbb F^2(0, T;H) $ for each $ p, z
    \in H $;
  \item there exists a positive constant $ \mathcal M $ such that, $ \mathbb P
    $-almost surely for almost every $ t \in [0,T] $,
    \begin{equation} 
      \label{eq:f-Lips}
      \nm{f(t,p_1,z_1) - f(t,p_2,z_2)}_{H} \leqslant
      \mathcal M(\nm{p_1-p_2}_{H} + \nm{z_1-z_2}_{H})
    \end{equation}
    for all $ p_1, p_2, z_1, z_2 \in H $.
\end{itemize}
The above notations are defined later in \cref{sec:pre}.

Bismut \cite{Bismut1973,Bismut1978} proposed the finite-dimensional linear
backward stochastic differential equations (BSDEs for short) and used them to
form the stochastic maximum principle for finite-dimensional stochastic optimal
control problems. Bensoussan \cite{Bensoussan1983} used the infinite-dimensional
linear BSDEs to form the maximum stochastic principle of stochastic distributed
parameter systems. Later, Pardoux and Peng \cite{Pardoux_peng1990} made a
significant breakthrough by establishing the well-posedness of a class of
finite-dimensional nonlinear BSDEs, and soon Hu and Peng \cite{Hu_Peng_1991}
proposed a highly non-trivial extension to the infinite-dimensional BSDEs.
Since then the theory of BSDEs began to develop quickly, mainly motivated by
applications to stochastic optimal control, partial differential equations and
mathematical finance; see \cite{Karoui1997,Ma_Yong1999,Pardoux2014,Peng1993,
Yong1999} and the references therein. We particularly refer the reader to
\cite[Chapter 6]{Fabbri2017} and \cite{LuZhangbook2021} and the references
therein for the applications of the infinite-dimensional BSDEs to the stochastic
optimal control problems.

The above mentioned works on the BSDEs all require that the filtration is
generated by the underlying Wiener process. Motivated by the transposition
method for non-homogeneous boundary value problems for partial differential
equations (see \cite{Lions1972I}), L\"u and Zhang \cite{LuZhang2013,Lu2014}
proposed a new notion of solution, the transposition solution, to BSDEs with
general filtration. The transposition solution coincides with the usual strong
solution when the filtration is natural. The transposition solution has been
successfully used to investigate the stochastic maximal principle for the
infinite-dimensional distributed parameter systems; see
\cite{LuZhang2019,LuZhang2021} and the references therein.

By now, the numerical solutions of the finite-dimensional BSDEs have been
extensively studied; see \cite{Bouchard2004,Chassagneux2014,Chassagneux2015,
E2019,Hu_Dualart_2011,Imkeller2010,Ma1994,
Zhang2004} and the references therein. In particular, we note that, based on the
definition of the transposition solution, Wang and Zhang \cite{Wang_Zhang_2011}
proposed a numerical method for solving finite-dimensional BSDEs. In view of the
fact that the numerical analysis of the infinite-dimensional stochastic
differential equations (SDEs for short) differs considerably from that of the
finite-dimensional SDEs, generally it is difficult to extend the numerical analysis of
finite-dimensional BSDEs to the infinite-dimensional BSDEs directly. We
note that there is a huge list of papers in the literature on the numerical
analysis of the infinite-dimensional SDEs; see
\cite{Barth2012,Beccari2019,Brehier2019,
Brehier2020, Cao2017, Carelli2012, Cui_Hong_2019, Cui_Hong_sun2021,
Cui_Hong_sun2021b,Du2002,Feng2017, Feng2021,
Kruse2014, Krusebook2014, Yan2005, Zhangbook2017}
and the references
therein. Despite this fact, the numerical analysis of the infinite-dimensional
nonlinear BSDEs is expected to be a challenging problem, since the SDEs and the
BSDEs are essentially different.


So far, the numerical analysis of the infinite-dimensional BSDEs is very
limited. Wang \cite{WangY2016} analyzed a semi-discrete Galerkin scheme for a
backward semilinear stochastic parabolic equation. Since this scheme uses the
eigenvectors of the Laplace operator, its application appears to be limited. Li
and Tang \cite{Li_Tang_2021} developed a splitting-up method for solving
backward stochastic partial differential equations.
To our best knowledge, no convergence rate is available for the spatial
semi-discretization with finite element method of backward semilinear stochastic
parabolic equations with general filtration.







In this paper, we use the notion of transposition solution introduced by L\"u
and Zhang \cite{LuZhang2013,Lu2014}, and make the following threefold main
contributions.
\begin{itemize}
\item Firstly, higher regularity of the transposition solution is derived,
which is essential for deriving convergence rate of the spatial
semi-discretization. We note that \cite{Lu2014} gives the basic regularity result
\[
  (p,z) \in D_\mathbb F([0,T];L^2(\Omega;H)) \times
  L_\mathbb F^2(0,T;H).
\]
Using the regularity estimates of the deterministic backward parabolic equations,
we prove that
\[
  (p,z) \in \big(
    L_\mathbb F^2(0,T;\dot H^2) \cap
    D_\mathbb F([0,T];L^2(\Omega;\dot H^1))
  \big) \times
  L_\mathbb F^2(0,T;\dot H^1).
\]
\item Secondly, for a spatial semi-discretization of \cref{eq:model}, we derive the
error estimate
\[
  \sup_{0 \leqslant t \leqslant T}
  \ssnm{(p-p_h)(t)}_{H} +
  \ssnm{p-p_h}_{L^2(0,T;\dot H^1)} +
  \ssnm{z-z_h}_{L^2(0,T;H)} \leqslant ch,
\]
which is optimal with respect to the regularity of the transposition solution.
This spatial semi-discretization adopts the standard continuous piecewise linear
finite element method. The case that $ \mathbb F $ is the natural filtration of
$ W(\cdot) $ is also covered by our numerical analysis, since in this case the
transposition solution coincides with the usual strong solution.

\item Thirdly, with the derived higher order regularity of the transposition
  solution and the convergence estimate, the first-order accuracy is derived for
  a spatially semi-discrete stochastic linear quadratic control problem, where
  the filtration is general and the diffusion term of the state equation
  contains the control variable. Here we note that the stochastic optimal
  control problems governed by stochastic partial differential equations have
  been extensively studied in the past four decades; however, these problems
  have rarely been numerically studied. To our best knowledge, this paper
  provides the first convergence rate for a spatial semi-discretization of a
  general stochastic linear quadratic control problem with general filtration.
\end{itemize}
We believe that the obtained theoretical results in this paper are useful for
further numerical analysis of backward semilinear stochastic parabolic equations
and stochastic linear quadratic control problems.

The rest of this paper is organized as follows. \cref{sec:pre} introduces the
preliminaries. \cref{sec:transposition} investigates the higher regularity of
the transposition solution to \cref{eq:model}. In \cref{sec:discretization}, we
derive the first-order spatial accuracy for a spatial semi-discretization of
\cref{eq:model}. Finally, using the derived higher regularity result and the
convergence estimate, we establish the convergence of a spatially semi-discrete
stochastic linear quadratic control problem, and provide some numerical results  in \cref{sec:application}.

\section{Preliminaries}
\label{sec:pre} 
Assume that $ X $ is a separable Hilbert space with norm $ \nm{\cdot}_X $. We
simply write the Hilbert space $ L^2(\Omega, \mathcal F_T, \mathbb P; X) $ as $
L^2(\Omega;X) $, and denote by $ \ssnm{\cdot}_X $ its norm. For any $ v \in
L^2(\Omega;X) $, we use $ \mathbb E v $ and $ \mathbb E_t v $ to denote
respectively the expectation of $ v $ and the conditional expectation of
$ v $ with respect to $ \mathcal F_t $ for each $ 0 \leqslant t \leqslant T $.
Let $ L_\mathbb F^2(0,T;X) $ be the space of all $ \mathbb F $-progressively
measurable processes $ \varphi $ such that
\[ 
  \ssnm{\varphi}_{L^2(0,T;X)} := \Big(
    \int_0^T \ssnm{\varphi(t)}_X^2 \, \mathrm{d}t
  \Big)^{1/2} < \infty.
\]
For any $ 0 < t < T $, the space $ L_\mathbb F^2(0,t;X) $ is defined analogously
to $ L_\mathbb F^2(0,T;X) $. Let $ L_\mathbb F^2(\Omega;C([0,T];X)) $ be the
space of all $ \mathbb F $-progressively measurable processes $ \varphi $ with
continuous paths in $ X $ such that
\[ 
  \ssnm{\varphi}_{C([0,T];X)} := \Big(
    \mathbb E \sup_{t \in [0,T]}
    \nm{\varphi(t)}_X^2
  \Big)^{1/2} < \infty.
\]
Let $ D_\mathbb F([0,T]; L^2(\Omega;X)) $ be the
space of all $ X $-valued and $
\mathbb F $-adapted processes that are right continuous with left limits in $
L^2(\Omega;X) $ with respect to the time variable. This is a Banach space with
the norm
\[
  \nm{\varphi}_{D_\mathbb F([0,T];L^2(\Omega; X))}
  := \max_{t \in [0,T]} \ssnm{\varphi(t)}_X
  \quad \forall \varphi \in D_\mathbb F([0,T];L^2(\Omega;X)).
\]

Denote $ H:=L^2(\mathcal O) $. For each $ \gamma \geqslant 0 $, define
\[
  \dot H^\gamma := \{
    (-\Delta)^{-\gamma/2} v \mid  v \in H 
  \}
\]
and endow this space with the norm
\[
  \nm{v}_{\dot H^\gamma} := \nm{(-\Delta)^{\gamma/2}v}_{H 
  }
  \quad \forall v \in \dot H^\gamma.
\]
We use $ \dot H^{-\gamma} $ to denote the dual space of $ \dot H^\gamma $. The
operator $ \Delta $ can be extended as a bounded linear operator from $
H $ to $ \dot H^{-2} $ by
\[
  \dual{\Delta v, \varphi}_{\dot H^2} =
  \int_\mathcal O v \Delta \varphi
  \quad \text{for all } v \in H
  \text{ and } \varphi \in \dot H^2,
\]
where $ \dual{\cdot,\cdot}_{\dot H^2} $ denotes the duality pairing between $
\dot H^{-2} $ and $ \dot H^2 $.

For any $ g \in L_\mathbb F^2(0,T;\dot H^\gamma) $ with $ -2 \leqslant \gamma <
\infty $, let $ S_0g $ be the mild solution of the stochastic parabolic equation
\begin{equation} 
  \begin{cases}
    \label{eq:y-g}
    \mathrm{d}y(t) = \Delta y(t) \, \mathrm{d}t +
    g(t) \, \mathrm{d}W(t) \quad \forall t \in [0,T], \\
    y(0) = 0.
  \end{cases}
\end{equation}
It is standard that (see, e.g., \cite[Chapter 3]{Gawarecki2011}) 
for any $ 0 \leqslant t \leqslant T $,
\begin{equation} 
  \label{eq:S0}
  (S_0g)(t) = \int_0^t e^{(t-s)\Delta} g(s) \, \mathrm{d}W(s)
  \quad\mathbb P \text{-a.s.}
\end{equation}
Moreover, a routine argument with It\^o's formula gives
\begin{equation} 
  \label{eq:S0-stab}
  \ssnm{(S_0g)(t)}_{\dot H^\gamma}^2 +
  2\ssnm{S_0g}_{L^2(0,t;\dot H^{\gamma+1})}^2 =
  \ssnm{g}_{L^2(0,t;\dot H^\gamma)}^2
  \quad \forall 0 < t \leqslant T.
\end{equation}


Finally, we introduce the mild solutions to a forward parabolic equation and a
backward parabolic equation, respectively. For any $ g \in L^2(0,T;H) $, let $
S_1g $ and $ S_2g $ be the mild solutions of the equations
\[ 
  \begin{cases}
    y'(t) = \Delta y(t) + g(t) \quad \forall t \in [0,T], \\
    y(0) = 0
  \end{cases}
\]
and
\[
  \begin{cases}
    z'(t) = -\Delta z(t) - g(t) \quad \forall t \in [0,T], \\
    z(T) = 0,
  \end{cases}
\]
respectively. We have (see, e.g., \cite[Chapter 3]{Yagi2010})
\begin{align} 
  (S_1g)(t) = \int_0^t e^{(t-s)\Delta} g(s) \, \mathrm{d}s
  \quad \forall t \in [0,T], \\
  (S_2g)(t) = \int_t^T e^{(s-t)\Delta} g(s) \, \mathrm{d}s
  \quad \forall t \in [0,T].
\end{align}
It is standard that, for any $ v, w \in L^2(0,T;H) $,
\begin{equation}
  \label{eq:S1-S2}
  (S_1v, w)_{L^2(0,T;H)} =
  (v, S_2w)_{L^2(0,T;H)},
\end{equation}
where $ (\cdot,\cdot)_{L^2(0,T;H)} $ denotes the inner product of the Hilbert
space $ L^2(0,T;H) $.

\section{Regularity}
\label{sec:transposition}
Following \cite{Lu2014}, we call
\[
  (p,z) \in D_\mathbb F([0,T];L^2(\Omega; H))
  \times L_\mathbb F^2(0,T;H)
\]
a transposition solution to \cref{eq:model} if
\begin{equation} 
  \label{eq:p-z}
  \begin{aligned}
    & \int_t^T \big[ p(s), g(s) \big] +
    \big[ z(s),\sigma(s) \big] \, \mathrm{d}s +
    \big[ p(t), v \big]\\
    ={} &
    \quad \int_t^T \big[
      f(s,p(s),z(s)),
      (S_0\sigma + S_1g)(s) + e^{(s-t)\Delta}v
    \big] \, \mathrm{d}s  {} \\
    &+ 
    \big[
      (S_0\sigma + S_1g)(T) + e^{(T-t)\Delta}v, p_T
    \big]
  \end{aligned}
\end{equation}
for all $ 0 \leqslant t \leqslant T $, $ (g,\sigma) \in \big( L_\mathbb F^2(0,T;
H) \big)^2 $ and $ v \in L^2(\Omega, \mathcal F_t, \mathbb P; H) $, where $
[\cdot,\cdot] $ denotes the inner product in $ L^2(\Omega; H)
$. For the unique existence of the transposition solution to \cref{eq:model},
we refer the reader to \cite[Theorem 3.1]{Lu2014}. Moreover, the proof of
\cite[Theorem 3.1]{Lu2014} contains that, for any $ 0 \leqslant t \leqslant T $,
\begin{equation}
  \label{eq:trans-p-int}
  p(t) = \mathbb E_t \Big(
    \int_t^T e^{(s-t)\Delta} f(s,p(s), z(s))
    \, \mathrm{d}s + e^{(T-t)\Delta} \, p_T
  \Big) \quad \mathbb P \text{-a.s.}
\end{equation}
In particular,
\begin{equation}
  \label{eq:pT}
  p(T) = p_T \quad \mathbb P \text{-a.s.}
\end{equation}



The main result of this section is the following theorem.
\begin{theorem} 
  \label{thm:regu}
  The transposition solution $ (p,z) $ of \cref{eq:model} possesses the
  following properties:
  \begin{enumerate}[(i)]
    \item $ p \in L_\mathbb F^2(0,T;\dot H^2) $;
    \item $ p $ admits a modification in $ D_\mathbb F([0,T];L^2(\Omega;\dot
      H^1)) $;
    \item $ z \in L_\mathbb F^2(0,T;\dot H^1) $.
  \end{enumerate}
\end{theorem}

\begin{remark}
  When $ \mathbb F $ is the natural filtration of $ W(\cdot) $, by the theory in
  \cite{Guatteri2005,Hu_Peng_1991} we easily obtain
  \[
    (p,z) \in \big(
      L_\mathbb F^2(0,T;\dot H^2) \cap
      L_\mathbb F^2(\Omega;C([0,T];\dot H^1))
    \big) \times
    L_\mathbb F^2(0,T;\dot H^1).
  \]
\end{remark}

The rest of this section is devoted to proving the above theorem. We first
introduce three technical lemmas, i.e. \cref{lem:cadlag,lem:eta-cont,eq:bee-regu}.
\begin{lemma}
  \label{lem:cadlag}
  Assume that $ 0 \leqslant a < b \leqslant T $ and $ v \in L^2(\Omega; X) $,
  with $ X $ being a separable Hilbert space. Then there exists a unique $ y \in
  D_\mathbb F([a,b];L^2(\Omega;X)) $ such that
  \begin{equation}
    \label{eq:cadlag}
    \mathbb P(y(t) = \mathbb E_t v) = 1
    \quad \forall t \in [a,b],
  \end{equation}
  where $ D_\mathbb F([a,b];L^2(\Omega;X)) $ is defined analogously to $
  D_\mathbb F([0,T];L^2(\Omega;X)) $.
\end{lemma}

\begin{lemma}
  \label{lem:eta-cont}
  Assume that $ w \in L^2(\Omega; C([0,T];X)) $, where $ X $ is a separable
  Hilbert space. Then
  \begin{equation}
    \label{eq:eta-cont}
    \lim_{m \to \infty} \sup_{
      \substack{
        0 \leqslant r < s \leqslant T \\
        s - r \leqslant 1/m
      }
    } \ssnm{w(r) - w(s)}_X = 0.
  \end{equation}
\end{lemma}

\begin{lemma} 
  \label{eq:bee-regu}
  Assume that $ v \in \dot H^1 $ and $ g \in L^2(0,T;H) $. Define
  \[
    w(t) := e^{(T-t)\Delta} v +
    \int_t^T e^{(s-t)\Delta} g(s) \, \mathrm{d}s
    \quad \forall t \in [0,T].
  \]
  Then
  \begin{equation*}
    \nm{w}_{C([0,T];\dot H^1)} +
    \nm{w}_{L^2(0,T;\dot H^2)} \leqslant
    C \big( \nm{v}_{\dot H^1} + \nm{g}_{L^2(0,T;H)} \big),
  \end{equation*}
  where $ C $ is a positive constant independent of $ v $, $ g $, $ T $ and $
  \mathcal O $.
\end{lemma}

The proofs of \cref{lem:cadlag,lem:eta-cont} are straightforward, and
\cref{eq:bee-regu} is standard; see, e.g., \cite[Theorem 10.11]
{Brezis2010}.

Based on the three lemmas above, we are in a position to prove \cref{thm:regu} as follows.

\medskip\noindent{\bf Proof of \cref{thm:regu}}. Let us first prove $ (i)
  $. Inserting $ t=0 $, $ \sigma = 0 $ and $ v = 0 $ into \cref{eq:p-z}, by
  \cref{eq:S1-S2} we obtain
  \begin{align*}
    \int_0^T [p(s), g(s)] \, \mathrm{d}s =
    \int_0^T [\eta(s), g(s)] \, \mathrm{d}s
    \quad \forall g \in L_\mathbb F^2(0,T;H),
  \end{align*}
  where
  \[
    \eta(s) := \int_s^T e^{(r-s)\Delta} f(r,p(r),z(r)) \, \mathrm{d}r +
    e^{(T-s)\Delta} p_T \quad \forall s \in [0,T].
  \]
  It follows that
  \begin{equation}
    \label{eq:p-eta}
    p = \mathcal E_\mathbb F \eta
    \quad\text{ in } L_\mathbb F^2(0,T;H),
  \end{equation}
  where $ \mathcal E_\mathbb F $ is the $ L^2(\Omega;L^2(0,T;H))
  $-orthogonal projection onto $ L_\mathbb F^2(0,T;H) $. For each $ n > 0
  $, define
  \begin{equation}
    \label{eq:etan}
    \eta_n(t) := \begin{cases}
      p_T & \text{ if } t = T, \\
      \frac{n}T \int_{\frac{jT}n}^{\frac{(j+1)T}n} \eta(t) \, \mathrm{d}t
      & \text{ if } t \in \big[ \frac{jT}n, \frac{(j+1)T}n \big)
      \text{ with } 0 \leqslant j < n.
    \end{cases}
  \end{equation}
  By \cref{eq:bee-regu} we have
  \begin{equation}
    \label{eq:eta-regu}
    \eta \in L^2(\Omega;L^2(0,T;\dot H^2)) \cap
    L^2(\Omega;C([0,T];\dot H^1)),
  \end{equation}
  and a routine density argument yields
  \begin{equation}
    \label{eq:426}
    \lim_{n \to \infty} \eta_n =
    \eta \quad\text{ in }
    L^2(\Omega;L^2(0,T;\dot H^2)).
  \end{equation}
  By \cref{lem:cadlag} we conclude that there exists a unique $ p_n \in
  D_\mathbb F([0,
  T]; L^2(\Omega;\dot H^2)) $ satisfying that
  \begin{equation}
    \label{eq:pn-etan}
    \mathbb P\big( p_n(t) = \mathbb E_t\eta_n(t) \big) = 1
    \quad \forall t \in [0,T].
  \end{equation}
  Hence, by the inequality
  \[
    \ssnm{(p_m - p_n)(t)}_{\dot H^2} =
    \ssnm{\mathbb E_t(\eta_m-\eta_n)(t)}_{\dot H^2}
    \leqslant \ssnm{(\eta_m - \eta_n)(t)}_{\dot H^2}
  \]
 for any $ m,n > 0 $ and $ t \in [0,T] $, we obtain
  \[
    \lim_{\substack{m \to \infty \\ n \to \infty}}
    \ssnm{p_m - p_n}_{L^2(0,T;\dot H^2)}^2 \leqslant
    \lim_{\substack{m \to \infty \\ n \to \infty}}
    \ssnm{\eta_m - \eta_n}_{L^2(0,T;\dot H^2)}^2
    \, \mathrm{d}t = 0 \quad \text{(by \cref{eq:426})}.
  \]
  It follows that $ \{p_n\}_{n=1}^\infty $ is a Cauchy sequence in $ L_\mathbb
  F^2(0,T;\dot H^2) $, and so there exists a unique $ \widetilde p \in L_\mathbb
  F^2(0, T;\dot H^2) $ such that
  \begin{equation}
    \label{eq:pn2wtp}
    \lim_{n \to \infty} p_n = \widetilde p
    \quad \text{ in } L_\mathbb F^2(0,T;\dot H^2).
  \end{equation}
  By \cref{eq:pn-etan}, it is easy to verify that, for each $ n > 0 $,
  \[
    \mathcal E_\mathbb F \eta_n = p_n
    \quad \text{ in } L_\mathbb F^2(0,T;H),
  \]
  so that by \cref{eq:p-eta,eq:426} we get
  \begin{align}
    \label{eq:427}
    \lim_{n \to \infty} p_n = \lim_{n \to \infty}
    \mathcal E_\mathbb F \eta_n =
    \mathcal E_\mathbb F \eta = p \quad
    \text{ in } L_\mathbb F^2(0,T;H).
  \end{align}
  In view of \cref{eq:pn2wtp,eq:427}, we readily obtain $ p \in L_\mathbb
  F^2(0,T;\dot H^2) $.

  Secondly, let us prove $ (ii) $. For any $ 0 < m < n < \infty $, we have
  \begin{align*}
    & \nm{p_m - p_n}_{
      D_\mathbb F([0,T];L^2(\Omega;\dot H^1))
    } \\
    ={} &
    \max_{0 \leqslant t \leqslant T}
    \ssnm{\mathbb E_t(\eta_m(t) - \eta_n(t))}_{\dot H^1}
    \quad \text{(by \cref{eq:pn-etan})} \\
    \leqslant{} &
    \max_{0 \leqslant t \leqslant T}
    \ssnm{\eta_m(t) - \eta_n(t)}_{\dot H^1} \\
    \leqslant{} &
    \max_{
      \substack{0 \leqslant r < s \leqslant T \\ s-r \leqslant 2T/m}
    } \ssnm{\eta(r) - \eta(s))}_{\dot H^1}
    \quad\text{(by \cref{eq:etan}),}
  \end{align*}
  so that \cref{lem:eta-cont} implies
  \[
    \lim_{\substack{m \to \infty \\ n \to \infty}}
    \nm{p_m - p_n}_{D_\mathbb F([0,T];L^2(\Omega;\dot H^1))} = 0.
  \]
  It follows that $ \{p_n\}_{n=1}^\infty $ is a Cauchy sequence in $ D_\mathbb
  F([0, T]; L^2(\Omega;\dot H^1)) $. Hence, there exists a unique $ \bar p \in
  D_\mathbb F([0,T];L^2(\Omega;\dot H^1)) $ such that
  \begin{equation}
    \label{eq:428}
    \lim_{n \to \infty} p_n = \bar p
    \quad\text{ in } D_\mathbb F([0,T];L^2(\Omega;\dot H^1)),
  \end{equation}
  which, together with \cref{eq:427}, yields
  \[
    p = \bar p \quad \text{ in }
    L_\mathbb F^2(0,T;H).
  \]
  It follows that
  \[
    p(t) = \bar p(t) \quad \text{ in }
    L^2(\Omega, \mathcal F_t, \mathbb P; H)
    \quad \text{a.e.}~t \in [0,T].
  \]
  Since $ p \in D_\mathbb F([0,T];L^2(\Omega;H)) $ and $ \bar p \in
  D_\mathbb F([0,T];L^2(\Omega;\dot H^1)) $, we then obtain
  \begin{equation}
    \label{eq:p-barp<T}
    p(t) = \bar p(t) \quad \text{ in }
    L^2(\Omega, \mathcal F_t, \mathbb P;H)
    \quad \forall t \in [0,T).
  \end{equation}
  Since \cref{eq:etan,eq:pn-etan} imply $ p_n(T) = p_T $, $ \mathbb P $-a.s., by
  \cref{eq:428} we get $ \bar p(T) = p_T $, $ \mathbb P $-a.s., and so
  \cref{eq:pT} implies
  \[ 
    p(T) = \bar p(T) \quad \text{ in }
    L^2(\Omega, \mathcal F_T, \mathbb P; H).
  \]
  By virtue of this equality and \cref{eq:p-barp<T}, we conclude that
  $ \bar p $ is exactly the modification of $ p $ in $ D_\mathbb F([0,T];
  L^2(\Omega;\dot H^1)) $.

  Thirdly, let us prove $ (iii) $. It is standard that there exists an
  orthonormal basis $ \{\phi_k\}_{k=0}^\infty \subset \dot H^2 $ of $ H $
  satisfying that
  \[
    -\Delta \phi_k = \lambda_k \phi_k,
  \]
  where $ \{\lambda_k\}_{k=0}^\infty $ is a nondecreasing sequence of strictly
  positive numbers with limit $ +\infty $. For each $ n \in \mathbb N $, define
  \begin{align*}
    z_n(t) &:= \sum_{k=0}^n (z(t),\phi_k)_{H} \phi_k,
    \quad 0 \leqslant t \leqslant T, \\
    F_n(t) &:= \sum_{k=0}^n (f(t,p(t),z(t)),\phi_k)_{H} \phi_k,
    \quad 0 \leqslant t \leqslant T,
  \end{align*}
  where $ (\cdot,\cdot)_{H} $ is the inner product of $ H $. For
  any $ 0 < m < n < \infty $, define $ \delta_{m,n} := z_n - z_m $.  Inserting $
  t = 0 $, $ g = 0 $, $ \sigma = -\Delta \delta_{m,n} $ and $ v =
  0 $ into \cref{eq:p-z} yields
  \begin{small}
  \begin{align*} 
    & \int_0^T \big[
      z, -\Delta \delta_{m,n}
    \big] \, \mathrm{d}s \\
    ={} &
    \int_0^T \big[
      f(s,p(s),z(s)), -(S_0\Delta\delta_{m,n})(s)
    \big] \, \mathrm{d}s +
    \big[ -(S_0\Delta\delta_{m,n})(T), p_T \big] \\
    ={} &
    \int_0^T \big[
      (F_n \!-\! F_m)(s), -(S_0\Delta\delta_{m,n})(s)
    \big] \, \mathrm{d}s +
    \Big[
      -(S_0\Delta \delta_{m,n})(T),
      \sum_{k=m+1}^n (p_T,\phi_k)_{H} \phi_k
    \Big] \\
    \leqslant{} &
    \ssnm{F_n-F_m}_{L^2(0,T;H)}
    \ssnm{
      S_0\Delta \delta_{m,n}
    }_{L^2(0,T;H)} + {} \\
    & \qquad \ssnm{
      (S_0\Delta \delta_{m,n})(T)
    }_{\dot H^{-1}}
    \Big(
      \sum_{k=m+1}^n \lambda_k \ssnm{(p_T, \phi_k)_{H}}_\mathbb R^2
    \Big)^{1/2} \\
    \leqslant{} &
    \ssnm{\delta_{m,n}}_{L^2(0,T;\dot H^{1})} \Big(
      \ssnm{F_n - F_m}_{L^2(0,T;H)} + \Big(
        \sum_{k={m+1}}^n \lambda_k \ssnm{(p_T,\phi_k)}_{\mathbb R}^2
      \Big)^{1/2}
    \Big),
  \end{align*}
  \end{small}
  where we have used   \cref{eq:S0-stab} and the equality
  \[
    \ssnm{\Delta \delta_{m,n}}_{L^2(0,T;\dot H^{-1})} =
    \ssnm{\delta_{m,n}}_{L^2(0,T;\dot H^1)}.
  \]
  Hence, by the equality
  \[
    \int_0^T \big[
      z(t), -\Delta \delta_{m,n}(t)
    \big] \, \mathrm{d}t =
    \ssnm{\delta_{m,n}}_{L^2(0,T;\dot H^1)}^2,
  \]
  we get
  \[
    \ssnm{\delta_{m,n}}_{L^2(0,T;\dot H^1)} \leqslant
    \ssnm{F_n - F_m}_{L^2(0,T;H)} + \Big(
      \sum_{k=m+1}^n \lambda_k \ssnm{(p_T,\phi_k)_{H}}_{\mathbb R}^2
    \Big)^{1/2}.
  \]
  Since
  \begin{align*}
    \lim_{\substack{m \to \infty \\ n \to \infty}}
    \ssnm{F_n - F_m}_{L^2(0,T;H)} + \Big(
      \sum_{k=m+1}^n \lambda_k \ssnm{(p_T,\phi_k)_{H}}_{\mathbb R}^2
    \Big)^{1/2} = 0,
  \end{align*}
  we obtain
  \[
    \lim_{\substack{m \to \infty \\ n \to \infty}}
    \ssnm{\delta_{m,n}}_{L^2(0,T;\dot H^1)} = 0.
  \]
  This implies that $ \{z_n\}_{n=0}^\infty $ is a Cauchy sequence in $ L_\mathbb
  F^2(0,T;\dot H^1) $. Hence, there exists a unique $ \widetilde z \in L_\mathbb
  F^2(0, T;\dot H^1) $ such that
  \[
    \lim_{n \to \infty} z_n = \widetilde z \quad \text{ in }
    L_\mathbb F^2(0,T;\dot H^1).
  \]
  By definition, we also have
  \[
    \lim_{n \to \infty} z_n = z \quad \text{ in }
    L_\mathbb F^2(0,T;H).
  \]
  Consequently, we obtain $ z \in L_\mathbb F^2(0,T;\dot H^1) $ and thus
  conclude the proof.
\hfill\ensuremath{
  \vbox{\hrule height0.6pt\hbox{%
    \vrule height1.3ex width0.6pt\hskip0.8ex
    \vrule width0.6pt}\hrule height0.6pt
  }
}

\section{Spatial semi-discretization}
\label{sec:discretization}
Let $ \mathcal K_h $ be a conventional conforming, shape regular and
quasi-uniform triangulation of $ \mathcal O $ consisting of $ d $-simplexes, and
let $ h $ denote the maximum diameter of the elements in $ \mathcal K_h $.
Define
\begin{align*}
  \mathcal V_h &:= \left\{
    v_h \in C(\overline{\mathcal O}) \mid\,
    v_h \text{ is linear on each } K \in \mathcal K_h
    \, \text{ and } v_h = 0 \text{ on }
    \partial\mathcal O
  \right\}.
\end{align*}
Let $ Q_h $ be the $ L^2(\mathcal O) $-orthogonal projection onto $ \mathcal V_h
$, and define the discrete Laplace operator $ \Delta_h: \mathcal V_h \to
\mathcal V_h $ by
\[
  \int_{\mathcal O} (\Delta_h v_h) w_h \, \mathrm{d}x =
  - \int_{\mathcal O} \nabla v_h \cdot \nabla w_h
  \, \mathrm{d}x \quad \text{ for all }
  v_h, w_h \in \mathcal V_h.
\]
For each $ \gamma \in \mathbb R $, let $ \dot H_h^\gamma $ be the space $
\mathcal V_h $ endowed with the norm
\[
  \nm{v_h}_{\dot H_h^\gamma} :=
  \nm{(-\Delta_h)^{\gamma/2} v_h}_{H}
  \quad \forall v_h \in \mathcal V_h.
\]
For any $ g \in L_\mathbb F^2(0,T;H) $, let $ S_0^hg $ and $ S_1^hg $ be the
mild solutions of the equations
\begin{equation}
  \label{eq:S0h}
  \begin{cases}
    \mathrm{d}y_h(t) = \Delta_h y_h(t) \, \mathrm{d}t +
    Q_hg(t) \, \mathrm{d} W(t)
    \quad \forall 0 \leqslant t \leqslant T, \\
    y_h(0) = 0
  \end{cases}
\end{equation}
and
\begin{equation}
  \begin{cases}
    \mathrm{d}y_h(t) = (\Delta_h y_h + Q_hg)(t) \, \mathrm{d}t
    \quad \forall 0 \leqslant t \leqslant T, \\
    y_h(0) = 0,
  \end{cases}
\end{equation}
respectively. It is standard that, for any $ 0 \leqslant t \leqslant T $,
\begin{align}
  (S_0^hg)(t) &= \int_0^t e^{(t-s)\Delta_h} Q_hg(s) \, \mathrm{d}W(s)
  \quad\mathbb P \text{-a.s.},
  \label{eq:S0h-int} \\
  (S_1^hg)(t) &= \int_0^t e^{(t-s)\Delta_h} Q_h g(s) \, \mathrm{d}s
  \quad\mathbb P \text{-a.s.}
  \label{eq:S1h-int}
\end{align}
In the rest of this paper, $ c $ denotes a generic positive constant independent
of $ h $, and its value may differ in different places.

We consider the following spatial semi-discretization of equation
\cref{eq:model}:
\begin{equation}
  \label{eq:semi}
  \begin{cases}
    \mathrm{d}p_h(t) = -\big(
      \Delta_h p_h(t) + Q_hf(t,p_h(t),z_h(t))
    \big) \, \mathrm{d}t +
    z_h(t) \, \mathrm{d}W(t), \quad 0 \leqslant t \leqslant T, \\
    p_h(T) = Q_hp_T.
  \end{cases}
\end{equation}
Similarly to \cref{eq:model}, this equation has a unique transposition
solution
\[
  (p_h,z_h) \in D_\mathbb F([0,T];L^2(\Omega;\mathcal V_h))
  \times L_\mathbb F^2(0,T;\mathcal V_h),
\]
which is defined by
\begin{equation} 
  \label{eq:ph-zh}
  \begin{aligned}
    & \int_t^T \big[ p_h(s), g_h(s) \big] +
    \big[ z_h(s),\sigma_h(s) \big] \, \mathrm{d}s +
    \big[ p_h(t), v_h \big]\\
    ={} &
   \quad  \int_t^T \big[
      f(s,p_h(s),z_h(s)),
      (S_0^h\sigma_h + S_1^hg_h)(s) + e^{(s-t)\Delta_h}v_h
    \big] \, \mathrm{d}s {} \\
    &
    + \big[
      (S_0^h\sigma_h + S_1^hg_h)(T) + e^{(T-t)\Delta_h}v_h, p_T
    \big]
  \end{aligned}
\end{equation}
for all $ 0 \leqslant t \leqslant T $, $ (g_h,\sigma_h) \in \big( L_\mathbb
F^2(0,T; \mathcal V_h) \big)^2 $ and $ v_h \in L^2(\Omega, \mathcal F_t, \mathbb
P; \mathcal V_h) $. Similarly to \cref{eq:trans-p-int}, we have, for any $ 0
\leqslant t \leqslant T $,
\begin{equation}
  \label{eq:trans-ph-int}
  p_h(t) = \mathbb E_t \Big(
    \int_t^T e^{(s-t)\Delta_h}
    Q_h f(s,p_h(s), z_h(s))
    \, \mathrm{d}s + e^{(T-t)\Delta_h} Q_h p_T
  \Big) \quad \mathbb P \text{-a.s.}
\end{equation}

The main result of this section is the following error estimate.
\begin{theorem}
  \label{thm:conv}
  Let $ (p,z) $ and $ (p_h,z_h) $ be the transposition solutions of
  \cref{eq:model,eq:semi}, respectively. Then
  \begin{equation}
    \label{eq:conv}
    \sup_{0 \leqslant t \leqslant T}
    \ssnm{(p-p_h)(t)}_{H} +
    \ssnm{p-p_h}_{L^2(0,T;\dot H^1)} +
    \ssnm{z-z_h}_{L^2(0,T;H)} \leqslant ch.
  \end{equation}
\end{theorem}

\subsection{Some auxiliary estimates}


\begin{lemma}
  \label{lem:1002}
  For any $ v \in \dot H^1 $, we have
  \begin{align}
    \max_{0 \leqslant t \leqslant T}
    \nm{
      (e^{t\Delta} - e^{t\Delta_h}Q_h)v
    }_{ H } &
    \leqslant ch \nm{v}_{\dot H^1},
    \label{eq:etDelta-etDeltah} \\
    \Big(
      \int_0^T \nm{
        (e^{t\Delta} - e^{t\Delta_h} Q_h)v
      }_{\dot H^1}^2
      \, \mathrm{d}t
    \Big)^{1/2} & \leqslant ch \nm{v}_{\dot H^1}.
    \label{eq:foo-3}
  \end{align}
\end{lemma}

\begin{lemma}
  \label{lem:1000}
  For any $ g \in L^2(0,T;H) $, we have
  \begin{align}
    \max_{0 \leqslant t \leqslant T} \Nm{
      \int_t^T \big(
        e^{(s-t)\Delta} - e^{(s-t)\Delta_h} Q_h
      \big) g(s) \, \mathrm{d}s
    }_{H} & \leqslant ch \nm{g}_{L^2(0,T;H)},
    \label{eq:back_deter_inf} \\
    \Big(
      \int_0^T \Nm{
        \int_t^T \big(
          e^{(s-t)\Delta} - e^{(s-t)\Delta_h}Q_h
        \big) g(s) \, \mathrm{d}s
      }_{\dot H^1}^2 \, \mathrm{d}t
    \Big)^{1/2}
    & \leqslant ch \nm{g}_{L^2(0,T;H)}.
    \label{eq:foo-2}
  \end{align}
\end{lemma}

\begin{lemma}
  \label{lem:1001}
  For any $ g_h \in L^2(0,T;\mathcal V_h) $, we have
  \begin{equation}
    \label{eq:foo-1}
    \Big(
      \int_0^T \Big\|
        \int_t^T e^{(s-t)\Delta_h} g_h(s) \, \mathrm{d}s
      \Big\|_{\dot H_h^1}^2 \, \mathrm{d}t
    \Big)^{1/2}
    \leqslant \nm{g_h}_{L^2(0,T;\dot H_h^{-1})}.
  \end{equation}
\end{lemma}


\noindent For the proof of \cref{eq:etDelta-etDeltah}, we refer the reader to
\cite[Theorem 3.5]{Thomee2006}. The proofs of
\cref{eq:foo-3,eq:back_deter_inf,eq:foo-2} are
similar to that of \cite[Lemma 3.6]{Thomee2006}. The inequality
\cref{eq:foo-1} can be proved by a routine energy argument.

\begin{lemma} 
  \label{lem:S0-S0h}
  For any $ g \in L_\mathbb F^2(0,T;H) $, we have
  \begin{equation}
    \label{eq:S0-S0h}
    \ssnm{(S_0g - S_0^hg)(T)}_{\dot H^{-1}} +
    \ssnm{(S_0 - S_0^h)g}_{L^2(0,T;H)}
     \leqslant c h \ssnm{g}_{L^2(0,T;H)}.
  \end{equation}
\end{lemma}
\begin{proof}
  Let $ y := S_0g $ and $ y_h := S_0^h g $. By \cref{eq:y-g} we have
  \[
    \begin{cases}
      \mathrm{d}Q_h y(t) = Q_h\Delta y \, \mathrm{d}t +
      Q_hg(t) \, \mathrm{d}W(t) \quad \forall t \in [0,T], \\
      Q_hy(0) = 0,
    \end{cases}
  \]
  and so from \cref{eq:S0h} we conclude that
  \[
    \begin{cases}
      \mathrm{d}e_h(t) = \Delta_h e_h(t) \, \mathrm{d}t +
      (\Delta_h Q_hy - Q_h \Delta y)(t) \, \mathrm{d}t,
      \quad \forall t \in [0,T], \\
      e_h(0) = 0,
    \end{cases}
  \]
  where $ e_h := y_h - Q_h y $. It is standard that
  \begin{align*}
    & \ssnm{e_h(T)}_{\dot H_h^{-1}} +
    \ssnm{e_h}_{L^2(0,T;\dot H_h^0)} \\
    \leqslant{} &
    c \ssnm{\Delta_h Q_h y - Q_h \Delta y}_{L^2(0,T;\dot H_h^{-2})} \\
    ={} &
    c \ssnm{
      \Delta_h(Q_h y - \Delta_h^{-1}Q_h \Delta y)
    }_{L^2(0,T;\dot H_h^{-2})} \\
    ={} &
    c \ssnm{
      Q_h y - \Delta_h^{-1}Q_h \Delta y
    }_{L^2(0,T;\dot H_h^{0})}.
  \end{align*}
  Hence,
  \begin{align*}
    & \ssnm{(y-y_h)(T)}_{\dot H^{-1}} +
    \ssnm{y - y_h}_{L^2(0,T;H)} \\
    \leqslant{} &
    \ssnm{(y-Q_hy)(T)}_{\dot H^{-1}} +
    c\ssnm{y-Q_hy}_{L^2(0,T;H)} +
    c\ssnm{y - \Delta_h^{-1}Q_h\Delta y}_{L^2(0,T;H)} \\
    \leqslant{} &
    c h \ssnm{y(T)}_{H} + ch \ssnm{y}_{L^2(0,T;\dot H^1)},
  \end{align*}
  by the following two standard estimates (see, e.g.,
  \cite[Theorems 4.4.20 and 5.7.6]{Brenner2008}):
  \begin{align*}
    \nm{(I-Q_h)v}_{\dot H^{-1}} \leqslant c h \nm{v}_H
    \quad \forall v \in H, & \\
    \nm{(I-Q_h)v}_{H} + \nm{v - \Delta_h^{-1}Q_h\Delta v}_H
    \leqslant c h \nm{v}_{\dot H^1}
    \quad \forall v \in \dot H^1. &
  \end{align*}
  Therefore, the desired estimate \cref{eq:S0-S0h} follows from
  \[
    \ssnm{y(T)}_{H} + \ssnm{y}_{L^2(0,T;\dot H^1)}
    \leqslant c \ssnm{g}_{L^2(0,T;H)},
  \]
  which can be obtained by inserting $ \gamma = 0 $ and $ t = T $ into
  \cref{eq:S0-stab}. This completes the proof.
\end{proof}

\subsection{Proof of Theorem 4.1}
To be specific, in this proof $ c $ denotes a positive constant depending only
on $ f $, $ p_T $, $ \mathcal O $, $ T $ and the regularity parameters of $
\mathcal K_h $. Let
\[
  e_h^p := p_h - p, \quad
  e_h^z := z_h - z.
\]
By $ f(\cdot, 0, 0) \in L_\mathbb F^2(0,T;H) $, \cref{eq:f-Lips} and
\[
  (p,z) \in D_\mathbb F([0,T];L^2(\Omega; H))
  \times L_\mathbb F^2(0,T;H),
\]
we have
\begin{equation}
  \label{eq:f(p,z)-esti}
  f(\cdot, p(\cdot), z(\cdot)) \in L_\mathbb F^2(0,T;H).
\end{equation}
We divide the proof into the following four steps.

{\it Step 1}. Let us prove, for any $ 0 \leqslant t < T $,
\begin{equation} 
  \label{eq:ehz-step1}
  \ssnm{e_h^z}_{L^2(t,T;H)}
  \leqslant c \Big(
    h + \sqrt{T-t} \big(
      \ssnm{e_h^p}_{L^2(t,T;H)} +
      \ssnm{e_h^z}_{L^2(t,T;H)}
    \big)
  \Big).
\end{equation}
To this end, let $ 0 \leqslant t < T $ be arbitrary but fixed. Define
\begin{equation}
  \label{eq:sigmah}
  \sigma_h(s) := \begin{cases}
    0 & \text{ if } 0 \leqslant s < t, \\
    (z_h-Q_hz)(s) & \text{ if } t \leqslant s \leqslant T.
  \end{cases}
\end{equation}
By \cref{eq:S0h-int,eq:sigmah} we get $ \mathbb P $-a.s.
\begin{equation}
  \label{eq:S0hsigmah}
  (S_0^h \sigma_h)(s) =
  \begin{cases}
    0 & \text{ if } 0 \leqslant s \leqslant t, \\
    \int_t^s e^{(r-t)\Delta_h} (z_h-Q_hz)(r) \, \mathrm{d}W(r)
    & \text{ if } t < s \leqslant T.
  \end{cases}
\end{equation}
It follows that, for any $ t < s \leqslant T $,
\begin{align}
  \ssnm{(S_0^h\sigma_h)(s)}_{H} &=
  \ssnm{
    \int_t^s e^{(r-t)\Delta_h}(z_h - Q_hz)(r) \, \mathrm{d}W(r)
  }_{H} \notag \\
  &= \Big(
    \int_t^s \ssnm{e^{(r-t)\Delta_h} (z_h-Q_hz)(r)}_{H}^2
    \, \mathrm{d}r
  \Big)^{1/2} \notag \\
  &\leqslant \Big(
    \int_t^s \ssnm{(z_h-Q_hz)(r)}_{H}^2
    \, \mathrm{d}r
  \Big)^{1/2} \notag \\
  & \leqslant \ssnm{z_h - Q_hz}_{L^2(t,T;H)}.
  \label{eq:S0hsigmah-esti}
\end{align}
Inserting $ g = 0 $, $ \sigma = \sigma_h $ and $ v=0 $ into
\cref{eq:p-z} yields
\begin{align*} 
  \int_t^T \big[ Q_hz(s), \, \sigma_h(s) \big] \, \mathrm{d}s =
  \int_t^T \big[ (S_0\sigma_h)(s), \, f(s,p(s),z(s)) \big] \, \mathrm{d}s +
  \big[ (S_0\sigma_h)(T), \, p_T \big],
\end{align*}
and inserting $ g_h = 0 $ and $ v_h = 0 $ into \cref{eq:ph-zh} gives
\begin{align*} 
  \int_t^T \big[ z_h(s), \, \sigma_h(s) \big] \, \mathrm{d}s =
  \int_t^T \big[ (S_0^h\sigma_h)(s), \, f(s,p_h(s),z_h(s)) \big] \, \mathrm{d}s +
  \big[ (S_0^h\sigma_h)(T), \, p_T \big].
\end{align*}
Combining the two equalities above yields
\begin{equation}
  \label{eq:zh-Qhz-trans}
  \ssnm{z_h-Q_hz}_{L^2(t,T;H)}^2 =
  \int_t^T [(z_h-Q_hz)(s), \sigma_h(s)] \, \mathrm{d}s =
  \mathbb I_1 + \mathbb I_2 + \mathbb I_3,
\end{equation}
where
\begin{align*}
  \mathbb I_1 &:=
  \int_t^T \big[
    (S_0^h\sigma_h)(s), \, f(s,p_h(s),z_h(s)) - f(s,p(s),z(s))
  \big] \, \mathrm{d}s, \\
  \mathbb I_2 &:= \int_t^T \big[
    (S_0^h\sigma_h - S_0\sigma_h)(s), \,
    f(s,p(s),z(s))
  \big] \, \mathrm{d}s, \\
  \mathbb I_3 &:= \big[ (S_0^h\sigma_h - S_0\sigma_h)(T), \, p_T \big].
\end{align*}
For $ \mathbb I_1 $ we have
\begin{small}
\begin{align*}
  \mathbb I_1 & \leqslant
  \int_t^T \ssnm{(S_0^h\sigma_h)(s)}_{H}
  \ssnm{f(s,p_h(s),z_h(s)) - f(s,p(s),z(s))}_{H} \, \mathrm{d}s \\
  & \leqslant
  \ssnm{z_h-Q_hz}_{L^2(t,T;H)}
  \int_t^T
  \ssnm{f(s,p_h(s),z_h(s)) - f(s,p(s),z(s))}_{H} \, \mathrm{d}s
  \quad\text{(by \cref{eq:S0hsigmah-esti})} \\
  & \leqslant
  c\ssnm{z_h-Q_hz}_{L^2(t,T;H)}
  \int_t^T
  \ssnm{e_h^p(s)}_{H} \, \mathrm{d}s +
  \ssnm{e_h^z(s)}_{H} \, \mathrm{d}s
  \quad\text{(by \cref{eq:f-Lips})} \\
  & \leqslant
  c \sqrt{T-t} \ssnm{z_h-Q_hz}_{L^2(t,T;H)}
  \big(
    \ssnm{e_h^p}_{L^2(t,T;H)} +
    \ssnm{e_h^z}_{L^2(t,T;H)}
  \big).
\end{align*}
\end{small}
For $ \mathbb I_2 $ we have
\begin{align*} 
  \mathbb I_2 & \leqslant
  \ssnm{(S_0^h-S_0)\sigma_h}_{L^2(t,T;H)}
  \ssnm{f(\cdot,p(\cdot),z(\cdot))}_{L^2(t,T;H)} \\
  & \leqslant c h \ssnm{\sigma_h}_{L^2(t,T;H)}
  \ssnm{f(\cdot,p(\cdot),z(\cdot)}_{L^2(t,T;H)}
  \quad\text{(by \cref{eq:S0-S0h})} \\
  &= ch \ssnm{z_h-Q_hz}_{L^2(t,T;H)}
  \ssnm{f(\cdot,p(\cdot),z(\cdot)}_{L^2(t,T;H)}
  \quad\text{(by \cref{eq:sigmah})} \\
  & \leqslant ch \ssnm{z_h - Q_hz}_{L^2(t,T;H)}
  \quad \text{(by \eqref{eq:f(p,z)-esti})}.
\end{align*}
For $ \mathbb I_3 $ we have
\begin{align*}
  \mathbb I_3 & \leqslant
  \ssnm{(S_0^h\sigma_h - S_0\sigma_h)(T)}_{\dot H^{-1}}
  \ssnm{p_T}_{\dot H^1} \\
  & \leqslant c h \ssnm{\sigma_h}_{L^2(0,T;H)}
  \ssnm{p_T}_{\dot H^1} \quad\text{(by \cref{eq:S0-S0h})} \\
  & \leqslant c h \ssnm{\sigma_h}_{L^2(0,T;H)} \\
  & = ch \ssnm{z_h-Q_hz}_{L^2(t,T;H)}
  \quad\text{(by \cref{eq:sigmah})}.
\end{align*}
Combining \cref{eq:zh-Qhz-trans} and the above estimates of $ \mathbb I_1 $, $
\mathbb I_2 $ and $ \mathbb I_3 $, we obtain
\begin{align*}
  &\ssnm{z_h-Q_hz}_{L^2(t,T;H)}^2 \\
   \leqslant &
  c \Big(
    h + \sqrt{T-t} \ssnm{e_h^p}_{L^2(t,T;H)} +
    \sqrt{T-t} \ssnm{e_h^z}_{L^2(t,T;H)}
  \Big)
  \ssnm{z_h-Q_hz}_{L^2(t,T;H)},
\end{align*}
which implies
\begin{equation}
  \label{eq:bj-1}
  \ssnm{z_h - Q_hz}_{L^2(t,T;H)} \leqslant
  ch + c \sqrt{T-t} \big(
    \ssnm{e_h^p}_{L^2(t,T;H)} +
    \ssnm{e_h^z}_{L^2(t,T;H)}
  \big).
\end{equation}
By \cref{thm:regu} we have $ z \in L_\mathbb F^2(0,T;\dot H^1) $, and so we have
the standard estimate
\begin{equation}
  \label{eq:bj-2}
  \ssnm{z - Q_hz}_{L^2(0,T;H)} \leqslant ch.
\end{equation}
Hence, the desired estimate \cref{eq:ehz-step1} follows from
\cref{eq:bj-1,eq:bj-2}.

{\it Step 2}. Let $ c^* $ be the particular constant $ c $ in the inequality
\cref{eq:ehz-step1}. Let $ t^* := \max\{0, T - 1/(2c^*)^2\} $, and so by
\cref{eq:ehz-step1} we get
\begin{equation}
  \label{eq:ehz}
  \ssnm{e_h^z}_{L^2(t,T;H)} \leqslant
  c \big( h + \ssnm{e_h^p}_{L^2(t,T;H)} \big)
  \quad \forall t \in [t^*,T].
\end{equation}
By \cref{eq:trans-p-int,eq:trans-ph-int} we have, for any
$ 0 \leqslant t \leqslant T $,
\[
  e_h^p(t) = \mathbb I_4 + \mathbb I_5 + \mathbb I_6
  \quad \mathbb P \text{-a.s.,}
\]
where
\begin{align*}
  \mathbb I_4 &:= \mathbb E_t \int_t^T e^{(s-t)\Delta_h}
  Q_h \big( f(s,p_h(s),z_h(s)) - f(s,p(s),z(s)) \big) \, \mathrm{d}s, \\
  \mathbb I_5 &:= \mathbb E_t \int_t^T
  \big(
    e^{(s-t)\Delta_h}Q_h - e^{(s-t)\Delta}
  \big) f(s,p(s),z(s)) \, \mathrm{d}s, \\
  \mathbb I_6 &:= \mathbb E_t\big(
    e^{(T-t)\Delta_h} Q_h - e^{(T-t)\Delta}
  \big) p_T.
\end{align*}
For $ \mathbb I_4 $ we have
\begin{align*}
  \ssnm{\mathbb I_4}_{H}^2 & \leqslant
  \ssnmB{
    \int_t^T e^{(s-t)\Delta_h} Q_h \big(
      f(s,p_h(s),z_h(s)) - f(s,p(s),z(s))
    \big) \, \mathrm{d}s
  }_{H}^2 \\
  & \leqslant
  \Big(
    \int_t^T \ssnm{
      e^{(s-t)\Delta_h} Q_h \big(
        f(s,p_h(s),z_h(s)) - f(s,p(s),z(s))
      \big)
    }_{H} \, \mathrm{d}s
  \Big)^2 \\
  & \leqslant
  \Big(
    \int_t^T \ssnm{
      f(s,p_h(s),z_h(s)) - f(s,p(s),z(s))
    }_{H} \, \mathrm{d}s
  \Big)^2 \\
  & \leqslant
  c \int_t^T \ssnm{
    f(s,p_h(s),z_h(s)) - f(s,p(s),z(s))
  }_{H}^2 \, \mathrm{d}s \\
  & \leqslant c \big(
    \ssnm{e_h^p}_{L^2(t,T;H)}^2 +
    \ssnm{e_h^z}_{L^2(t,T;H)}^2
  \big) \quad \text{(by \cref{eq:f-Lips}).}
\end{align*}
For $ \mathbb I_5 $ we have
\begin{align*} 
  \ssnm{\mathbb I_5}_{H}^2 &
  \leqslant \ssnmB{
    \int_t^T \big(
      e^{(s-t)\Delta_h} Q_h - e^{(s-t)\Delta}
    \big) f(s,p(s),z(s))
    \, \mathrm{d}s
  }_{H}^2 \\
  & \leqslant c h^2 \ssnm{f(\cdot,p(\cdot),z(\cdot))}_{
    L^2(t,T;H)
  }^2 \quad\text{(by \cref{eq:back_deter_inf})} \\
  & \leqslant ch^2 \quad\text{(by \eqref{eq:f(p,z)-esti})}.
\end{align*}
For $ \mathbb I_6 $ we have
\begin{align*}
  \ssnm{\mathbb I_6}_{H}^2
  & \leqslant \ssnm{
    \Big(e^{(T-t)\Delta_h} Q_h - e^{(T-t)\Delta}\Big) p_T
  }_{H}^2 \\
  & \leqslant c h^2 \ssnm{p_T}_{\dot H^1}^2
  \quad \text{(by \cref{eq:etDelta-etDeltah})} \\
  & \leqslant c h^2.
\end{align*}
Combining the above estimates of $ \mathbb I_4 $, $ \mathbb I_5 $ and $
\mathbb I_6 $, we obtain
\begin{equation}
  \label{eq:eve-2}
  \ssnm{e_h^p(t)}_{H}^2 \leqslant
  c \big(
    h^2 + \ssnm{e_h^p}_{L^2(t,T;H)}^2 +
    \ssnm{e_h^z}_{L^2(t,T;H)}^2
  \big) \quad \forall t \in [0,T],
\end{equation}
which, together with \cref{eq:ehz}, implies
\begin{align*}
  \ssnm{e_h^p(t)}_{H}^2 \leqslant
  c\big( h^2 + \ssnm{e_h^p}_{L^2(t,T;H)}^2 \big)
  \quad \forall t \in [t^*,T].
\end{align*}
Using the Gronwall's inequality then gives
\[
  \sup_{t^* \leqslant t \leqslant T}
  \ssnm{e_h^p(t)}_{H} \leqslant ch.
\]
Hence, by \cref{eq:ehz} we get
\begin{equation}
  \label{eq:eve}
  \sup_{t^* \leqslant t \leqslant T}
  \ssnm{e_h^p(t)}_{H} +
  \ssnm{e_h^z}_{L^2(t^*,T;H)} \leqslant
  ch.
\end{equation}

{\it Step 3}. Note that $ t^* $ depends only on $ f $, $ p_T $, $ \mathcal O $,
$ T $ and the regularity parameters of $ \mathcal K_h $. With the estimate $
\ssnm{e_h^p(t^*)}_H \leqslant c h $ (see \cref{eq:eve}) and similar arguments in
the proof of \cref{eq:ehz}, we get
\begin{equation}
  \label{eq:eve-3}
  \ssnm{e_h^z}_{L^2(t,t^*;H)} \leqslant
  c\big( h + \ssnm{e_h^p}_{L^2(t,t^*;H)} \big)
  \quad \forall t \in \big[ \max\{0,2t^*-T\}, \, t^* \big],
\end{equation}
which, together with \cref{eq:ehz}, implies
\[
  \ssnm{e_h^z}_{L^2(t,T;H)} \leqslant
  c\big( h + \ssnm{e_h^p}_{L^2(t,T;H)} \big)
  \quad \forall t \in \big[ \max\{0,2t^*-T\}, \, T \big].
\]
Combining the above estimate and \cref{eq:eve-2}, we obtain
\[
  \ssnm{e_h^p(t)}_H^2 \leqslant c\big(
    h^2 + \ssnm{e_h^p}_{L^2(t,T;H)}^2
  \big) \quad \forall t \in \big[ \max\{0,2t^*-T\}, \, T \big].
\]
The Gronwall's inequality then leads to
\[
  \sup_{\max\{0,2t^*-T\} \leqslant t \leqslant T}
  \ssnm{e_h^p(t)}_H \leqslant ch,
\]
and so by \cref{eq:eve-3} we obtain
\[
  \sup_{\max\{0,2t^*-T\} \leqslant t \leqslant T}
  \ssnm{e_h^p(t)}_H + \ssnm{e_h^z}_{L^2(\max\{0,2t^*-T\},T;H)}
  \leqslant ch.
\]
Repeating the above procedure several times finally yields
\begin{equation} 
  \label{eq:conv-1}
  \sup_{0 \leqslant t \leqslant T}
  \ssnm{e_h^p(t)}_{H} +
  \ssnm{e_h^z}_{L^2(0,T;H)} \leqslant ch.
\end{equation}

{\it Step 4}. Let us prove
\begin{equation}
  \label{eq:conv-h1}
  \ssnm{p-p_h}_{L^2(0,T;\dot H^1)} \leqslant ch.
\end{equation}
Set
\begin{align*}
  F(t) &:= f(t, p(t), z(t)), \quad 0 \leqslant t \leqslant T, \\
  F_h(t) &:= Q_h f(t, p_h(t), z_h(t)), \quad 0 \leqslant t \leqslant T.
\end{align*}
By \cref{eq:f-Lips,eq:conv-1}, we have
\begin{align}
  \ssnm{F_h - Q_hF}_{L^2(0,T;H)} &=
  \ssnm{
    Q_h \big(
      f(\cdot, p_h(\cdot), z_h(\cdot)) -
      f(\cdot,p(\cdot), z(\cdot))
    \big)
  }_{L^2(0,T;H)} \notag \\
  & \leqslant
  \ssnm{
    f(\cdot, p_h(\cdot), z_h(\cdot)) -
    f(\cdot,p(\cdot), z(\cdot))
  }_{L^2(0,T;H)} \notag \\
  & \leqslant c \big(
    \ssnm{e_h^p}_{L^2(0,T;H)} +
    \ssnm{e_h^z}_{L^2(0,T;H)}
  \big) \notag \\
  & \leqslant c h. \label{eq:Fh-QhF}
\end{align}
By \cref{eq:trans-p-int,eq:trans-ph-int} we obtain, for any $ t \in [0,T] $,
\begin{align*}
  (p - p_h)(t) = \mathbb E_t\big( \eta_1 + \eta_2 + \eta_3 \big)(t)
  \quad \mathbb P \text{-a.s.,}
\end{align*}
where
\begin{align*}
  \eta_1(t) &:= \int_t^T \big(
    e^{(s-t)\Delta} - e^{(s-t)\Delta_h} Q_h
  \big) F(s) \, \mathrm{d}s, \\
  \eta_2(t) &:= \int_t^T e^{(s-t)\Delta_h} (Q_hF - F_h)(s) \, \mathrm{d}s, \\
  \eta_3(t) &:= \big(
    e^{(T-t)\Delta} - e^{(T-t)\Delta_h} Q_h
  \big) p_T.
\end{align*}
It follows that
\begin{align}
  \ssnm{p-p_h}_{L^2(0,T;\dot H^1)} & \leqslant
  \ssnm{\eta_1 + \eta_2 + \eta_3}_{L^2(0,T;\dot H^1)} \notag \\
  & \leqslant
  \ssnm{\eta_1}_{L^2(0,T;\dot H^1)} +
  \ssnm{\eta_2}_{L^2(0,T;\dot H^1)} +
  \ssnm{\eta_3}_{L^2(0,T;\dot H^1)}.
  \label{eq:p-ph}
\end{align}
For $ \eta_1 $ we have
\begin{align*} 
  \ssnm{\eta_1}_{L^2(0,T;\dot H^1)} &=
  \Big(
    \mathbb E \int_0^T \nmB{
      \int_t^T \big(
        e^{(s-t)\Delta} - e^{(s-t)\Delta_h}Q_h
      \big) F(s) \, \mathrm{d}s
    }_{\dot H^1}^2 \, \mathrm{d}t
  \Big)^{1/2} \\
  & \leqslant
  c h \Big(
    \mathbb E \nm{F}_{L^2(0,T;H)}^2
  \Big)^{1/2} \quad \text{(by \cref{eq:foo-2})} \\
  & \leqslant ch \quad \text{(by \eqref{eq:f(p,z)-esti})}.
\end{align*}
For $ \eta_2 $ we have
\begin{align*}
  \ssnm{\eta_2}_{L^2(0,T;\dot H^1)} &=
  \Big(
    \mathbb E \int_0^T \nmB{
      \int_t^T e^{(s-t)\Delta_h} (Q_hF - F_h)(s) \, \mathrm{d}s
    }_{\dot H^1}^2 \, \mathrm{d}t
  \Big)^{1/2} \\
  & \leqslant \Big(
    \mathbb E \nm{F_h - Q_hF}_{L^2(0,T;H)}^2
  \Big)^{1/2} \quad \text{(by \cref{eq:foo-1})} \\
  & = \ssnm{F_h - Q_hF}_{L^2(0,T;H)} \\
  & \leqslant ch \quad \text{(by \cref{eq:Fh-QhF}).}
\end{align*}
For $ \eta_3 $ we have
\begin{align*}
  \ssnm{\eta_3}_{L^2(0,T;\dot H^1)} & =
  \Big(
    \mathbb E\int_0^T \nmb{
      (e^{(T-t)\Delta} - e^{(T-t)\Delta_h} Q_h)p_T
    }_{\dot H^1}^2 \, \mathrm{d}t
  \Big)^{1/2} \\
  & \leqslant
  ch \Big(
    \mathbb E \nm{p_T}_{\dot H^1}^2 \, \mathrm{d}t
  \Big)^{1/2} \quad \text{(by \cref{eq:foo-3})} \\
  & = ch \ssnm{p_T}_{\dot H^1} \leqslant ch.
\end{align*}
Combining \cref{eq:p-ph} and the above estimates of $ \eta_1 $, $ \eta_2 $ and $
\eta_3 $ yields \cref{eq:conv-h1}. Finally, summing up
\cref{eq:conv-1,eq:conv-h1} proves \cref{eq:conv} and thus concludes the proof
of \cref{thm:conv}.

\section{Application to a stochastic linear quadratic control problem}
\label{sec:application}
\subsection{Continuous problem}
We consider the following stochastic linear quadratic control problem:
\begin{equation} 
  \label{eq:optim}
  \min_{
    \substack{
      u \in L_\mathbb F^2(0,T;H) \\
      y \in L_\mathbb F^2(0,T;H)
    }
  } \frac12 \ssnm{y-y_d}_{L^2(0,T; H)}^2 +
  \frac{\nu}2 \ssnm{u}_{L^2(0,T; H)}^2,
\end{equation}
subject to the state equation
\begin{equation} 
  \label{eq:state}
  \begin{cases}
    \mathrm{d}y(t) = (\Delta y + \alpha_0 y + \alpha_1 u)(t) \, \mathrm{d}t +
    (\alpha_2 y + \alpha_3 u)(t) \, \mathrm{d}W(t),
    & 0 \leqslant t \leqslant T, \\
    y(0) = 0,
  \end{cases}
\end{equation}
where $ 0 < \nu, T < \infty $, $ \alpha_0, \alpha_1, \alpha_2, \alpha_3 \in
L^\infty(0,T) $, and $ y_d \in L_\mathbb F^2(0,T;H) $. It is standard that
(see \cite[Theorem 8.1]{Lu2014}) problem \cref{eq:optim} admits a unique
solution $ (\bar u, \bar y) $, and
\begin{equation}
  \label{eq:optim_cond}
  \bar u = - \nu^{-1}(\alpha_1 \bar p + \alpha_3 \bar z),
\end{equation}
where $ (\bar p, \bar z) $ is the transposition solution of the backward
stochastic parabolic equation
\begin{equation}
  \label{eq:barp-barz}
  \begin{cases}
    \mathrm{d}\bar p(t) = -(
    \Delta \bar p + \alpha_0 \bar p +
    \bar y - y_d + \alpha_2 \bar z
    )(t) \, \mathrm{d}t +
    \bar z(t) \, \mathrm{d}W(t), \quad 0 \leqslant t \leqslant T, \\
    \bar p(T) = 0.
  \end{cases}
\end{equation}
Since \cref{thm:regu} implies
\[
  \alpha_1 \bar p + \alpha_3 \bar z \in
  L_\mathbb F^2(0,T;\dot H^1),
\]
we then obtain
\begin{equation}
  \label{eq:baru-regu}
  \bar u \in L_\mathbb F^2(0,T; \dot H^1).
\end{equation}
Moreover, we have
\begin{equation}
  \label{eq:bary-regu}
  \bar y \in L_\mathbb F^2(\Omega;C([0,T];\dot H^1)) \cap L_\mathbb F^2(0,T;\dot H^2).
\end{equation}

\begin{remark}
  For the theoretical treatment of the stochastic linear quadratic
  control problems, we refer the reader to \cite{LuZhangbook2021} and
  the references therein.
\end{remark}

\begin{remark}
  The regularity result \cref{eq:bary-regu} is standard. It can be easily proved
  by the standard Galerkin method and the theory of the finite-dimensional linear
  SDEs (see \cite[Chapter 3]{Pardoux2014}).
\end{remark}


\begin{remark}
  We note that Zhou and Li~\cite{ZhouLi2021} established the convergence of a
  full discretization for a Neumann boundary control problem governed by a
  stochastic parabolic equation with additive boundary noise and general
  filtration. When $ \mathbb F $ is the natural filtration of the Brownian
  motion, we refer the reader to \cite{Dunst2016}, \cite{LiZhou2020} and
  \cite{Wang2020a} for some related numerical analysis.
\end{remark}
\subsection{Spatially semi-discrete problem}
The spatial semi-discretization of problem \cref{eq:optim} reads as follows:
\begin{small}
\begin{equation}
  \label{eq:optim_h}
  \min_{
    \substack{
      u_h \in L_\mathbb F^2(0,T;\mathcal V_h) \\
      y_h \in L_\mathbb F^2(0,T;\mathcal V_h)
    }
  } \frac12 \ssnm{y_h-y_d}_{L^2(0,T; H)}^2 +
  \frac{\nu}2 \ssnm{u_h}_{L^2(0,T; H)}^2,
\end{equation}
\end{small}
subject to the state equation
\begin{small}
\begin{equation}
  \label{eq:state_h}
  \begin{cases}
    \mathrm{d}y_h(t) =
    (\Delta_h y_h + \alpha_0 y_h + \alpha_1 u_h)(t) \, \mathrm{d}t +
    (\alpha_2 y_h + \alpha_3 u_h)(t) \, \mathrm{d}W(t),
    & 0 \leqslant t \leqslant T, \\
    y_h(0) = 0.
  \end{cases}
\end{equation}
\end{small}
Similarly to problem \cref{eq:optim}, problem \cref{eq:optim_h} admits a unique
solution $ (\bar u_h, \bar y_h) $, and
\begin{equation} 
  \label{eq:optim_cond_h}
  \bar u_h = - \nu^{-1} (
  \alpha_1 \bar p_h + \alpha_3\bar z_h
  ),
\end{equation}
where $ (\bar p_h, \bar z_h) $ is the transposition solution of the spatially
semi-discrete backward stochastic parabolic equation
\begin{small}
\begin{equation}
  \label{eq:barph-barzh}
  \begin{cases}
    \mathrm{d}\bar p_h(t) = -(
    \Delta_h \bar p_h + \alpha_0 p_h +
    \bar y_h - Q_hy_d + \alpha_2 \bar z_h
    )(t) \, \mathrm{d}t + \bar z_h(t) \, \mathrm{d}W(t),
    \, 0 \leqslant t \leqslant T, \\
    \bar p_h(T) = 0.
  \end{cases}
\end{equation}
\end{small}

The main result of this section is the following error estimate.
\begin{theorem} 
  \label{thm:optim-conv}
  Let $ (\bar u, \bar y) $ and $ (\bar u_h, \bar y_h) $ be the solutions of
  problems \cref{eq:optim,eq:optim_h}, respectively. Then
  \begin{equation}
    \label{eq:optim-conv}
    \ssnm{\bar u - \bar u_h}_{L^2(0,T;H)} +
    \ssnm{\bar y - \bar y_h}_{L^2(0,T;H)}
    \leqslant ch.
  \end{equation}
\end{theorem}

To prove this theorem, we first introduce three lemmas.
\begin{lemma}
  \label{lem:yh-stab}
  Let $ y_h $ be the solution of the stochastic equation
  \begin{equation}
    \label{eq:yh-stab-1}
    \begin{cases}
      \mathrm{d}y_h(t) =
      (\Delta_h y_h + \alpha_0 y_h + g_h)(t) \, \mathrm{d}t +
      \alpha_2(t) y_h(t) \, \mathrm{d}W(t),
      \quad 0 \leqslant t \leqslant T, \\
      y_h(0) = 0,
    \end{cases}
  \end{equation}
  where $ g_h \in L_\mathbb F^2(0,T;\mathcal V_h) $. Then
  \begin{equation}
    \label{eq:yh-stab}
    \ssnm{y_h}_{L^2(0,T;H)} \leqslant
    c \ssnm{g_h}_{L^2(0,T;\dot H_h^{-2})}.
  \end{equation}
\end{lemma}
\begin{proof} 
  Letting $ w_h := (-\Delta_h)^{-1/2} y_h $, by \cref{eq:yh-stab-1} we have
  \begin{small}
  \begin{equation*}
    \begin{cases}
      \mathrm{d}w_h(t) =
      \big( \Delta_h w_h + \alpha_0 w_h + (-\Delta_h)^{-1/2}g_h \big)(t)
      \, \mathrm{d}t + \alpha_2(t) w_h(t)
      \, \mathrm{d}W(t), \quad 0 \leqslant t \leqslant T, \\
      w_h(0) = 0.
    \end{cases}
  \end{equation*}
  \end{small}
  A routine argument with the It\^o's formula then yields, for any $ 0 \leqslant t
  \leqslant T $,
  \begin{align}
      & \ssnm{w_h(t)}_{H}^2 + 2\ssnm{w_h}_{L^2(0,t;\dot H_h^1)}^2 \notag \\
    ={} &
    2 \int_0^t \big[
      w_h(s),
      \alpha_0(s) w_h(s) + (-\Delta_h)^{-1/2}g_h(s)
    \big] \, \mathrm{d}s +
    \ssnm{\alpha_2 w_h}_{L^2(0,t;H)}^2 \notag \\
    \leqslant{} &
    c \ssnm{w_h}_{L^2(0,t;H)}^2 +
    2 \ssnm{w_h}_{L^2(0,t;\dot H_h^1)}
    \ssnm{(-\Delta_h)^{-1/2}g_h}_{L^2(0,t;\dot H_h^{-1})} \notag \\
    \leqslant{} &
    c \ssnm{w_h}_{L^2(0,t;H)}^2 +
    \ssnm{w_h}_{L^2(0,t;\dot H_h^1)}^2 +
    \ssnm{(-\Delta_h)^{-1/2}g_h}_{L^2(0,t;\dot H_h^{-1})}^2 \notag \\
    ={} &
    c \ssnm{w_h}_{L^2(0,t;H)}^2 +
    \ssnm{w_h}_{L^2(0,t;\dot H_h^1)}^2 +
    \ssnm{g_h}_{L^2(0,t;\dot H_h^{-2})}^2.
    \label{eq:fuck}
  \end{align}
  It follows that
  \begin{align*}
    \ssnm{w_h(t)}_{H}^2 \leqslant
    c\ssnm{w_h}_{L^2(0,t;H)}^2 +
    \ssnm{g_h}_{L^2(0,t;\dot H_h^{-2})}^2
    \quad \forall 0 \leqslant t \leqslant T,
  \end{align*}
  and so using the Gronwall's inequality yields
  \begin{equation}
    \label{eq:fuck-1}
    \sup_{0 \leqslant t \leqslant T}
    \ssnm{w_h(t)}_H \leqslant c
    \ssnm{g_h}_{L^2(0,T;\dot H_h^{-2})}.
  \end{equation}
  In addition, inserting $ t = T $ into \cref{eq:fuck} gives
  \begin{equation}
    \label{eq:fuck-2}
    \ssnm{w_h}_{L^2(0,T;\dot H_h^1)}^2 \leqslant
    c \ssnm{w_h}_{L^2(0,T;H)}^2 +
    \ssnm{g_h}_{L^2(0,T;\dot H_h^{-2})}^2.
  \end{equation}
  Finally, combining \cref{eq:fuck-1}, \cref{eq:fuck-2} and the fact
  \[
    \ssnm{y_h}_{L^2(0,T;H)} = \ssnm{w_h}_{L^2(0,T;\dot H_h^1)},
  \]
  we readily obtain \cref{eq:yh-stab}.
\end{proof}

\begin{lemma} 
  Assume that $ (\bar u, \bar y) $ is the solution of problem \cref{eq:optim}.
  Let $ y_h $ be the mild solution of the stochastic equation
  \begin{small}
  \begin{equation}
    \label{eq:yh-def}
    \begin{cases}
      \mathrm{d}y_h(t) \!=\!
      (\Delta_h y_h \!+\! \alpha_0 y_h \!+\! \alpha_1  Q_h\bar u)(t)
      \mathrm{d}t \!+\!
      (\alpha_2 y_h \!+\! \alpha_3 Q_h \bar u)(t)  \mathrm{d}W(t),
      \, 0 \leqslant t \leqslant T, \\
      y_h(0) = 0.
    \end{cases}
  \end{equation}
  \end{small}
  Then
  \begin{equation}
    \label{eq:bary-yh}
    \ssnm{\bar y - y_h}_{L^2(0,T;H)}
    \leqslant ch^2.
  \end{equation}
\end{lemma}
\begin{proof}
  By definition we have
  \begin{small}
  \begin{equation*}
    \begin{cases}
      \mathrm{d}\bar y(t) =
      (\Delta \bar y + \alpha_0 \bar y + \alpha_1 \bar u)(t) \, \mathrm{d}t +
      (\alpha_2 \bar y + \alpha_3 \bar u)(t) \, \mathrm{d}W(t),
      \quad 0 \leqslant t \leqslant T, \\
      \bar y(0) = 0,
    \end{cases}
  \end{equation*}
  \end{small}
  so that
  \begin{small}
  \begin{equation*}
    \begin{cases}
      \mathrm{d}Q_h\bar y(t) =
      Q_h(\Delta \bar y + \alpha_0\bar y + \alpha_1 \bar u)(t) \, \mathrm{d}t +
      Q_h (\alpha_2\bar y + \alpha_3 \bar u)(t) \, \mathrm{d}W(t),
      \quad 0 \leqslant t \leqslant T, \\
      Q_h\bar y(0) = 0.
    \end{cases}
  \end{equation*}
  \end{small}
  Hence, by \cref{eq:yh-def} we get
  \begin{small}
  \begin{equation*}
    \begin{cases}
      \mathrm{d}e_h(t) =
      (\Delta_h e_h + \alpha_0e_h +
      \Delta_h Q_h\bar y - Q_h \Delta\bar y)(t) \, \mathrm{d}t +
      \alpha_2(t) e_h(t)
      \, \mathrm{d}W(t), \quad 0 \leqslant t \leqslant T, \\
      e_h(0) = 0,
    \end{cases}
  \end{equation*}
  \end{small}
  where $ e_h := y_h - Q_h \bar y $. By \cref{lem:yh-stab} we then obtain
  \begin{align*}
    \ssnm{e_h}_{L^2(0,T;H)} & \leqslant c
    \ssnm{\Delta_h Q_h\bar y - Q_h \Delta\bar y}_{L^2(0,T;\dot H_h^{-2})} \\
    & = c\ssnm{\Delta_h(Q_h\bar y - \Delta_h^{-1} Q_h \Delta\bar y)}_{
      L^2(0,T;\dot H_h^{-2})
    } \\
    & = c\ssnm{Q_h\bar y - \Delta_h^{-1} Q_h \Delta\bar y}_{
      L^2(0,T;H)
    }.
  \end{align*}
  It follows that
  \begin{align*}
    & \ssnm{\bar y - y_h}_{L^2(0,T;H)} \\
    \leqslant{} &
    \ssnm{\bar y - Q_h\bar y}_{L^2(0,T;H)} +
    \ssnm{e_h}_{L^2(0,T;H)} \\
    \leqslant{} &
    \ssnm{\bar y - Q_h \bar y}_{L^2(0,T;H)} +
    c \ssnm{Q_h\bar y - \Delta_h^{-1}Q_h\Delta\bar y}_{L^2(0,T;H)} \\
    \leqslant{} &
    c\ssnm{\bar y - Q_h \bar y}_{L^2(0,T;H)} +
    c\ssnm{\bar y - \Delta_h^{-1}Q_h\Delta \bar y}_{L^2(0,T;H)}.
  \end{align*}
  Hence, the desired estimate \cref{eq:bary-yh} follows from \cref{eq:bary-regu}
  and the standard estimate
  \[
    \nm{v - Q_hv}_{H} +
    \nm{v - \Delta_h^{-1}Q_h \Delta v}_{H} \leqslant
    c h^2 \nm{v}_{\dot H^2} \quad \forall v \in \dot H^2.
  \]
  This completes the proof.
\end{proof}

\begin{lemma}
  \label{lem:lxy}
  Assume that $ g_h, v_h \in L_\mathbb F^2(0,T;\mathcal V_h) $.
  Let $ (p_h,z_h) $ be the transposition solution of the equation
  \begin{equation*} 
    \begin{cases}
      \mathrm{d}p_h(t) =
      -\big(
        \Delta_h p_h + \alpha_0 p_h + g_h + \alpha_2 z_h
      \big)(t) \, \mathrm{d}t +
      z_h(t) \, \mathrm{d}W(t), \, 0 \leqslant t \leqslant T, \\
      p_h(T) = 0,
    \end{cases}
  \end{equation*}
  and let $ y_h $ be the mild solution of the equation
  \begin{equation*}
    \begin{cases}
      \mathrm{d}y_h(t) \!=\!
      (\Delta_h y_h \!+\! \alpha_0 y_h \!+\! \alpha_1  v_h)(t)
      \mathrm{d}t \!+\!
      (\alpha_2 y_h \!+\! \alpha_3 v_h)(t)  \mathrm{d}W(t),
      \, 0 \leqslant t \leqslant T, \\
      y_h(0) = 0.
    \end{cases}
  \end{equation*}
  Then
  \begin{equation}
    \label{eq:lxy}
    \int_0^T \big[
      (\alpha_1 p_h + \alpha_3 z_h)(t), \, v_h(t)
    \big] \, \mathrm{d}t =
    \int_0^T \big[ g_h(t), \, y_h(t) \big] \, \mathrm{d}t.
  \end{equation}
\end{lemma}
\begin{proof} 
  Note that
  \[
    y_h = S_0^h(\alpha_2 y_h + \alpha_3 v_h) +
    S_1^h(\alpha_0 y_h + \alpha_1 v_h).
  \]
  By definition (see \cref{eq:ph-zh}), we then obtain
   \begin{align*}
     & \int_0^T \big[p_h(t), \, (\alpha_0 y_h + \alpha_1 v_h)(t) \big] +
     \big[z_h(t), \, (\alpha_2 y_h + \alpha_3 v_h)(t)\big] \, \mathrm{d}t \\
     ={} &
     \int_0^T \big[
       (\alpha_0 p_h + g_h + \alpha_2 z_h)(t), \, y_h(t)
     \big] \, \mathrm{d}t,
   \end{align*}
   which implies the desired equality \cref{eq:lxy}.
   This completes the proof.
\end{proof}

Finally, we are in a position to show the proof of \cref{thm:optim-conv} as
follows.

\medskip\noindent{\bf Proof of \cref{thm:optim-conv}}. Firstly, we present some
preliminary results. Let $ y_h $ be the mild solution of \cref{eq:yh-def}, let $
(\bar p, \bar z) $ and $ (\bar p_h, \bar z_h) $ be the transposition solutions
of \cref{eq:barp-barz,eq:barph-barzh}, respectively, and let $ (p_h,z_h) $ be
the transposition solution of the spatially semi-discrete backward stochastic
parabolic equation
  \begin{small}
  \begin{equation} 
    \begin{cases}
      \mathrm{d}p_h(t) =
      -\big(
        \Delta_h p_h + \alpha_0 p_h + Q_h(\bar y - y_d) + \alpha_2 z_h
      \big)(t) \, \mathrm{d}t +
      z_h(t) \, \mathrm{d}W(t), \, 0 \leqslant t \leqslant T, \\
      p_h(T) = 0.
    \end{cases}
  \end{equation}
  \end{small}
  By \cref{lem:lxy}, it is easy to verify that
  \begin{equation}
    \label{eq:barph-barph-ph-yh}
    \int_0^T \big[
      \alpha_1(p_h - \bar p_h) + \alpha_3(z_h-\bar z_h), \,
      \bar u_h - \bar u
    \big] \, \mathrm{d}t = \int_0^T
    [\bar y - \bar y_h, \, \bar y_h - y_h] \, \mathrm{d}t.
  \end{equation}
  By \cref{thm:conv} we have
  \begin{equation}
    \label{eq:barp-ph}
    \sup_{0 \leqslant t \leqslant T}
    \ssnm{(\bar p-p_h)(t)}_{H} +
    \ssnm{\bar p-p_h}_{L^2(0,T;\dot H^1)} +
    \ssnm{\bar z-z_h}_{L^2(0,T;H)} \leqslant ch.
  \end{equation}

  Secondly, we use the standard argument in the numerical analysis of
  optimization with PDE constraints (see \cite[Theorem 3.4]{Hinze2009}) to prove
  \cref{eq:optim-conv}. By \cref{eq:optim_cond} we obtain
  \begin{equation}
    \label{eq:500}
    \nu \int_0^T [\bar u, \bar u - \bar u_h] \, \mathrm{d}t =
    \int_0^T [
      \alpha_1 \bar p + \alpha_3 \bar z,
      \bar u_h - \bar u
    ] \, \mathrm{d}t,
  \end{equation}
  and by \cref{eq:optim_cond_h} we get
  \begin{equation}
    \label{eq:501}
    -\nu \int_0^T [\bar u_h, \bar u - \bar u_h] \, \mathrm{d}t
    = \int_0^T [
      \alpha_1 \bar p_h + \alpha_3 \bar z_h, \bar u - \bar u_h
    ] \, \mathrm{d}t.
  \end{equation}
  Summing up the above two equalities yields
  \begin{align*}
    \nu \ssnm{\bar u - \bar u_h}_{L^2(0,T;H)}^2
    & = \int_0^T \big[
      \alpha_1(\bar p - \bar p_h) + \alpha_3(\bar z - \bar z_h), \,
      \bar u_h - \bar u
    \big] \, \mathrm{d}t \\
    &= \mathbb I_1 + \mathbb I_2,
  \end{align*}
  where
  \begin{align*}
    \mathbb I_1 &:= \int_0^T \big[
      \alpha_1(\bar p - p_h) + \alpha_3(\bar z - z_h), \,
      \bar u_h - \bar u
    \big] \, \mathrm{d}t, \\
    \mathbb I_2 &:= \int_0^T \big[
      \alpha_1(p_h - \bar p_h) + \alpha_3(z_h - \bar z_h), \,
      \bar u_h - \bar u
    \big] \, \mathrm{d}t.
  \end{align*}
  For $ \mathbb I_1 $ we have
  \begin{align*}
    \mathbb I_1 & \leqslant c \big(
      \ssnm{\bar p - p_h}_{L^2(0,T;H)} +
      \ssnm{\bar z - z_h}_{L^2(0,T;H)}
    \big) \ssnm{\bar u - \bar u_h}_{L^2(0,T;H)} \\
    & \leqslant c h \ssnm{\bar u - \bar u_h}_{L^2(0,T;H)}
    \quad \text{(by \cref{eq:barp-ph}).}
  \end{align*}
  For $ \mathbb I_2 $ we have
  \begin{align*}
    \mathbb I_2 &= \int_0^T [
      \bar y - \bar y_h, \bar y_h - y_h
    ] \, \mathrm{d}t \quad \text{(by \cref{eq:barph-barph-ph-yh})} \\
    & =
    -\ssnm{\bar y - \bar y_h}_{L^2(0,T;H)}^2 +
    \int_0^T [\bar y - \bar y_h, \bar y - y_h] \, \mathrm{d}t \\
    & \leqslant
    -\frac12 \ssnm{\bar y - \bar y_h}_{L^2(0,T;H)}^2 +
    \frac12 \ssnm{\bar y - y_h}_{L^2(0,T;H)}^2 \\
    & \leqslant
    -\frac12 \ssnm{\bar y - \bar y_h}_{L^2(0,T;H)}^2 +
    ch^4 \quad\text{(by \cref{eq:bary-yh}).}
  \end{align*}
  Combining the above estimates of $ \mathbb I_1 $ and $ \mathbb I_2 $ yields
  \begin{align*}
    & \nu\ssnm{\bar u - \bar u_h}_{L^2(0,T;H)}^2 +
    \frac12 \ssnm{\bar y - \bar y_h}_{L^2(0,T;H)}^2 \\
    \leqslant{} &
    ch \ssnm{\bar u - \bar u_h}_{L^2(0,T;H)} + ch^4 \\
    \leqslant{} &
    ch^2 + \frac\nu2 \ssnm{\bar u - \bar u_h}_{L^2(0,T;H)}^2,
  \end{align*}
  which implies \cref{eq:optim-conv}. This completes the proof.
\hfill\ensuremath{
  \vbox{\hrule height0.6pt\hbox{%
    \vrule height1.3ex width0.6pt\hskip0.8ex
    \vrule width0.6pt}\hrule height0.6pt
  }
}

\subsection{Numerical results}
Let $ J \geqslant 2 $ be a positive integer. Set $ t_j := j\tau $ for each $
0 \leqslant j \leqslant J $, where $ \tau := T/J $. Define
\[ 
  L_{\mathbb F,\tau}^2(0,T;\mathcal V_h) := \big\{
    U \in L_\mathbb F^2(0,T;\mathcal V_h) \mid
    \text{$U$ is constant on $ [t_j,t_{j+1}) $, $ \forall 0 \leqslant j < J $}
  \big\}.
\]
For each $ 0 \leqslant j < J $, define $ \delta W_j := W(t_{j+1}) -
W(t_j) $, and let $ \mathrm{Net}_j: \mathbb R^d \times \mathbb R^j \to
\mathbb R $ be a neural network. Define the control $ U \in L_{\mathbb
F,\tau}^2(0,T;\mathcal V_h) $ as follows:
\begin{align*} 
  & \qquad\qquad
  U(0,x) := \mathrm{Net}_0(x) \quad
  \text{for each interior node $ x $ of $ \mathcal K_h $}; \\
  &\text{for any $ 1 \leqslant j < J $}, \\
  & \qquad\qquad
  U(t_j,x) := \mathrm{Net}_j(x,\delta W_0, \ldots, \delta W_{j-1}) \quad
  \text{for each interior node $ x $ of $ \mathcal K_h $}.
\end{align*}
The numerical optimal control $ \bar U $ is obtained by training these neural
networks $ \mathrm{Net}_j $ with the loss function
\[
  \mathcal L_{h,\tau} (U) :=
  \frac12
    \sum_{j=0}^{J-1} \ssnm{Y-y_d(t_j)}_{L^2(t_j,t_{j+1};H)}^2 +
  \frac\nu2 \ssnm{U}_{L^2(0,T;H)}^2,
\]
where the numerical state $ Y \in L_{\mathbb F,\tau}^2(0,T;\mathcal V_h) $ is
calculated as follows:
\begin{equation*} 
  \begin{cases}
    Y(t_{j+1}) - Y(t_j) =  \tau \Delta_h Y(t_{j+1}) +
    \tau (\alpha_0 Y + \alpha_1 U)(t_j)  {} \\
    \qquad\qquad\qquad\qquad\qquad\qquad
   + ( \alpha_2 Y + \alpha_3 U)(t_j) \delta W_j,
    \qquad 0 \leqslant j < J, \\
    Y(0) = 0.
  \end{cases}
\end{equation*}

Our numerical experiment adopts the following settings: $ \mathcal O = (0,1) $,
$
T=\num{2e-1} $, $ \nu = \num{1e-2} $, $ \alpha_0 = \alpha_1 = \alpha_2 = 1 $, $
\alpha_3 = \num{1e-1} $; each $ \mathrm{Net}_j $, $ 0 \leqslant j < J $, is a
fully connected feedforward neural network with four hidden layers,
where each hidden layer possesses $ 200 $ neurons with the ReLU activation
function; in the computation of each numerical optimal control, these neural
networks $ \mathrm{Net}_j $ are trained by the Adam optimization algorithm with
$ 5000 $ iterations, where each iteration uses $ 128 $ paths. The numerical
experiment is performed by means of PyTorch with single precision.

In \cref{tab:ex1,tab:ex2}, the reference control $ U^* $ is the numerical
control with $ h = 1/64 $, $ \bar Y $ and $ Y^* $ are the numerical states of
the controls $ \bar U $ and $ U^* $, respectively, and $ \ssnm{\cdot} $ denotes
the norm $ \ssnm{\cdot}_{L^2(0,T;H)} $. In addition, the norm $ \ssnm{\cdot}
_{L^2(0,T;H)} $ is calculated by $ \num{2.56e6} $ paths. The numerical results
in \cref{tab:ex1,tab:ex2} demonstrate that
\[
  \ssnm{\bar U - U^*}_{L^2(0,T;H)} +
  \ssnm{\bar Y - Y^*}_{L^2(0,T;H)}
\]
is close to $ O(h) $, which agrees well with \cref{thm:optim-conv}.

\begin{table}[H]
  \footnotesize \setlength{\tabcolsep}{2pt}
  \caption{ Numerical results with
    $ y_d(t,x) = x^{-0.49}, \,\, (t,x) \in [0,T] \times \mathcal O $.
  } \label{tab:ex1}
  \begin{center}
    \begin{tabular}{c@{\hspace{8pt}}cc@{\hspace{8pt}}ccc@{\hspace{10pt}}cc@{\hspace{8pt}}cc}
      \toprule
      & \multicolumn{4}{c}{$ J=50 $} & &
      \multicolumn{4}{c}{$ J=80 $} \\
      \cmidrule{2-5} \cmidrule{7-10}
      $ h $ & $ \ssnm{\bar U - U^*}$ & Order & $ \ssnm{\bar Y-Y^*} $ & Order & &
      $ \ssnm{\bar U-U^*}$ & Order & $ \ssnm{\bar Y-Y^*} $ & Order \\
      \hline
      $  1/4  $ & $ \num{3.12e-1} $ & --       & $\num{9.14e-3}$ & --   &  & $ \num{3.10e-1} $  & --       & $ \num{9.37e-3}$  & --   \\
      $ 1/8 $   & $ \num{1.43e-1} $ & $ 1.12 $ & $\num{4.80e-3}$ & 0.93 &  & $ \num{1.45e-1} $  & $ 1.10 $ & $ \num{4.80e-3} $ & 0.97 \\
      $ 1/16 $  & $ \num{6.82e-2} $ & $ 1.07 $ & $\num{1.95e-3}$ & 1.30 &  & $ \num{7.36e-2} $  & $ 0.98 $ & $ \num{2.29e-3} $ & 1.07 \\
      $ 1/32 $  & $ \num{3.87e-2} $ & $ 0.82 $ & $\num{9.37e-4}$ & 1.06 &  & $ \num{4.24e-2} $  & $ 0.80 $ & $ \num{1.07e-3} $ & 1.10 \\
      \bottomrule
    \end{tabular}
  \end{center}
\end{table}

\begin{table}[H]
  \footnotesize \setlength{\tabcolsep}{2pt}
  \caption{ Numerical results with 
    $ y_d(t,x) = \big( 1+W(t)^2 \big) x^{-0.49},
    \,\, (t,x) \in [0,T] \times \mathcal O $.
  } \label{tab:ex2}
  \begin{center}
    \begin{tabular}{c@{\hspace{8pt}}cc@{\hspace{8pt}}ccc@{\hspace{10pt}}cc@{\hspace{8pt}}cc}
      \toprule
      & \multicolumn{4}{c}{$ J=50 $} & &
      \multicolumn{4}{c}{$ J=80 $} \\
      \cmidrule{2-5} \cmidrule{7-10}
      $ h $ & $ \ssnm{\bar U - U^*}$ & Order & $ \ssnm{\bar Y-Y^*} $ & Order & &
      $ \ssnm{\bar U-U^*}$ & Order & $ \ssnm{\bar Y-Y^*} $ & Order \\
      \hline
      $ 1/4 $  & $ \num{3.69e-1} $ & --       & $\num{1.14e-2}$ & --   &  & $ \num{3.63e-1} $ & --       & $  \num{1.13e-2}$ & --   \\
      $ 1/8 $  & $ \num{1.66e-1} $ & $ 1.16 $ & $\num{5.97e-3}$ & 0.93 &  & $ \num{1.73e-1} $ & $ 1.07 $ & $ \num{5.92e-3} $ & 0.94 \\
      $ 1/16 $ & $ \num{8.38e-2} $ & $ 0.98 $ & $\num{2.97e-3}$ & 1.01 &  & $ \num{9.35e-2} $ & $ 0.89 $ & $ \num{2.93e-3} $ & 1.01 \\
      $ 1/32 $ & $ \num{4.53e-2} $ & $ 0.89 $ & $\num{1.29e-3}$ & 1.20 &  & $ \num{5.19e-2} $ & $ 0.85 $ & $ \num{1.41e-3} $ & 1.06 \\
      \bottomrule
    \end{tabular}
  \end{center}
\end{table}

\section{Conclusion}
For the backward semilinear stochastic parabolic equation with general
filtration, we have derived the higher regularity of the solution to the
continuous problem, and obtained the first-order accuracy of the spatial
semi-discretization with the linear finite element method. The derived
theoretical results have been applied to a general stochastic linear quadratic
control problem, and the first-order spatial accuracy has been derived for a
spatially semi-discrete stochastic linear quadratic control problem. 
\bibliographystyle{plain}

\end{document}